\RequirePackage{snapshot}
\documentclass[AMA,STIX1COL]{WileyNJD-v2}

\usepackage{cmap}
\usepackage{mathtools}
\usepackage{multirow}
\usepackage[english=british]{csquotes}
\usepackage{stmaryrd}
\usepackage[nameinlink,noabbrev]{cleveref}
\usepackage{pgfplots}
\pgfplotsset{compat=1.16}
\usepgfplotslibrary{groupplots}

\ifpdf
\hypersetup{
  pdfauthor          = {Matthias Kirchhart; Erick Schulz},
  pdftitle           = {Div-Curl Problems and H1-regular Stream Functions in 3D Lipschitz Domains},  
  pdfencoding        = {unicode},
  pdflang            = {en-GB},
  pdfdisplaydoctitle = {true},
  pdfkeywords        = {div-curl system, stream function, vector potential, regularity, vorticity},
  unicode            = {true},
  naturalnames       = false,
  hypertexnames      = false,
  breaklinks         = true,
  psdextra,hidelinks,final
}

\DeclareMathAlphabet\mathcalbf{OMS}{cmsy}{b}{n}

\DeclareMathOperator{\supp}{supp}
\DeclareMathOperator{\range}{Range}

\newcommand{\pd}{\partial}

\newcommand{\intdx}[4]{\int_{#1}^{#2} {#3}\,{\mathrm d}{#4}}

\newcommand{\bigO}[1]{\mathcal{O}\left({#1}\right)}
\newcommand{\reals}{\mathbb{R}}
\newcommand{\wholespace}{\reals^3}
\newcommand{\vv}[1]{\mathbf{#1}}     
\newcommand{\gv}[1]{\boldsymbol{#1}} 

\newcommand{\Newton}{\mathcal{N}}
\newcommand{\NewtonVec}{\mathcalbf{N}}
\DeclareMathOperator{\DIV}{div}
\DeclareMathOperator{\curl}{curl}
\newcommand{\grad}{\text{--}\nabla}
\newcommand{\gradp}{\nabla}
\newcommand{\Laplace}{\text{--}\Delta}

\newcommand{\Laplacevec}{\textbf{--}\gv{\Delta}}

\newcommand{\LebVec}{\vv{L}^2(\Omega)}
\newcommand{\Hcurl}{\vv{H}(\curl;\Omega)}

\newcommand{\HOcurl}{\vv{H}_0(\curl;\Omega)}
\newcommand{\HOcurldual}{\vv{H}_0(\curl;\Omega)'}
\newcommand{\HOne}{H^1(\Omega)}
\newcommand{\HOneZero}{H^1_0(\Omega)}
\newcommand{\HOneVec}{\vv{H}^1(\Omega)}
\newcommand{\HOneZeroVec}{\vv{H}^1_0(\Omega)}
\newcommand{\HOneDualVec}{\vv{H}^{-1}(\Omega)}

\newcommand{\Hloc}{H^1_{\mathrm{loc}}(\wholespace)}
\newcommand{\HlocVec}{\vv{H}^1_{\mathrm{loc}}(\wholespace)}

\newcommand{\Hdiv}{\vv{H}(\DIV;\Omega)}
\newcommand{\HOdiv}{\vv{H}_0(\DIV;\Omega)}
\newcommand{\distris}{\mathcal{D}'(\Omega)}
\newcommand{\distrisvec}{\mathcalbf{D}'(\Omega)}
\newcommand{\wholedistris}{\mathcal{D}'(\wholespace)}
\newcommand{\wholedistrisvec}{\mathcalbf{D}'(\wholespace)}
\newcommand{\testfcts}{\mathcal{D}(\Omega)}
\newcommand{\testfctsvec}{\mathcalbf{D}(\Omega)}
\newcommand{\wholetestfcts}{\mathcal{D}(\wholespace)}
\newcommand{\wholetestfctsvec}{\mathcalbf{D}(\wholespace)}

\newcommand{\wholesmoothfcts}{\mathcal{E}(\wholespace)}
\newcommand{\wholesmoothfctsvec}{\mathcalbf{E}(\wholespace)}
\newcommand{\compdistris}{\mathcal{E}'(\wholespace)}
\newcommand{\compdistrisvec}{\mathcalbf{E}'(\wholespace)}

\DeclareMathOperator{\tracegrad}{\text{--}\nabla_\Gamma}
\DeclareMathOperator{\tracegradp}{\nabla_\Gamma}
\DeclareMathOperator{\tracediv}{\mathrm{div}_\Gamma}
\DeclareMathOperator{\tracecurl}{\mathrm{curl}_\Gamma}
\DeclareMathOperator{\tracecurlvec}{\mathbf{curl}_\Gamma}
\DeclareMathOperator{\tracelaplace}{\text{--}\Delta_\Gamma}

\newcommand{\SobTraceDual}{H^{-\frac{1}{2}}(\Gamma)}

\newcommand{\HRMinusDiv}{\vv{H}^{-\frac{1}{2}}_{\mathrm{R}}(\tracediv;\Gamma)}
\newcommand{\HTMinusCurl}{\vv{H}^{-\frac{1}{2}}_{\mathrm{T}}(\tracecurl;\Gamma)}

\newcommand{\FEMcont}{S_h^n(\Omega)}
\newcommand{\FEMOcont}{S_{h,0}^n(\Omega)}
\newcommand{\FEMNED}{\mathbf{NED}_h^n(\Omega)}
\newcommand{\FEMNEDO}{\mathbf{NED}_{h,0}^n(\Omega)}
\newcommand{\FEMRT}{\mathbf{RT}_h^n(\Omega)}
\newcommand{\FEMRTO}{\mathbf{RT}_{h,0}^n(\Omega)}
\newcommand{\FEMdisc}{S_{h,\mathrm{disc}}^n(\Omega)}

\newcommand{\BEMcont}{S_h^n(\Gamma)}
\newcommand{\BEMRT}{\mathbf{RT}_h^n(\Gamma)}

\newcommand{\BEMdisc}{S_{h,\mathrm{disc}}^n(\Gamma)}

\articletype{Preprint}%

\received{21 May 2020}
\revised{01 February 2021}

\usepackage{bookmark}
\usepackage{microtype}

\begin{document}

\title{Div-Curl Problems and $\vv{H}^1$-regular Stream Functions in 3D Lipschitz Domains}
\author[1]{Matthias Kirchhart}
\author[2]{Erick Schulz}
\authormark{Matthias Kirchhart and Erick Schulz}
\address[1]{\orgdiv{Applied and Computational Mathematics}, \orgname{RWTH~Aachen}, \orgaddress{\country{Germany}}}
\address[2]{\orgdiv{Seminar in Applied Mathematics}, \orgname{ETH~Z{\"u}rich}, \orgaddress{\country{Switzerland}}}
\corres{Matthias Kirchhart, \email{kirchhart@acom.rwth-aachen.de}}
\presentaddress{Applied and Computational Mathematics\\RWTH Aachen\\Schinkelstra{\ss}e 2\\52062 Aachen\\ Germany}

\abstract[Abstract]{We consider the problem of recovering the divergence-free
velocity field $\vv{U}\in\LebVec$ of a given vorticity 	$\vv{F}=\curl\vv{U}$ on
a bounded Lipschitz domain $\Omega\subset\wholespace$. To that end, we solve the
\enquote{div-curl problem} for a given $\vv{F}\in\HOneDualVec$. The solution is
expressed in terms of a vector potential (or stream function) $\vv{A}\in\HOneVec$
such that $\vv{U}=\curl{\vv{A}}$. After discussing existence and uniqueness of
solutions and associated vector potentials, we propose a well-posed construction
for the stream function. A numerical method based on this construction is presented,
and experiments confirm that the resulting approximations display higher regularity
than those of another common approach.}


\maketitle

\section{Introduction}\label{sec: Introduction}
Let $\Omega\subset\wholespace$ be a bounded Lipschitz domain. Given a vorticity field
$\vv{F}\left(\mathbf{x}\right)\in\wholespace$ defined over $\Omega$, we are interested
in solving the problem of \emph{velocity recovery}: 
\begin{equation}\label{eqn:velocity-recovery}
\left\lbrace
\begin{aligned}
\curl\vv{U} &= \vv{F}\\
\DIV\vv{U}  &= 0
\end{aligned}
\right.
\qquad\text{in $\Omega$.}
\end{equation}

This problem naturally arises in fluid mechanics when studying the vorticity
formulation of the incompressible Navier--Stokes equations. Vortex methods,
for example, are based on the vorticity formulation and require a solution of 
problem~\eqref{eqn:velocity-recovery} at every time-step.\cite{cottet2000} 
While our motivation lies in fluid dynamics, this \enquote{div-curl problem}
also is interesting in its own right. 

On the whole space $\wholespace$, this problem is a classical matter. Whenever
$\vv{F}$ is smooth and compactly supported, the unique solution $\vv{U}$ of problem~%
\eqref{eqn:velocity-recovery} that decays to zero at infinity is given by
the Biot--Savart law.\cite[Proposition~2.16]{majda2001} However, the case where
$\Omega$ is a bounded domain is significantly more challenging.

In numerical simulations of the incompressible Navier--Stokes equations, it
is common to fulfil the constraint $\DIV\vv{U}=0$ only approximately, but it has
recently been demonstrated that such a violation can cause significant instabilities.
The importance for numerical methods to fulfil this constraint \emph{exactly} was
stressed by John et al.\cite{john2017} One way of achieving this requirement is
the introduction of a \emph{stream function}, or \emph{vector potential}: instead
of solving problem~\eqref{eqn:velocity-recovery} directly, one seeks an approximation
$\vv{A}_h$ of an auxiliary vector-field $\vv{A}$ such that $\vv{U}=\curl\vv{A}$.
Because of the vector calculus identity $\DIV\circ\curl\equiv 0$, the velocity
field $\vv{U}_h = \curl\vv{A}_h$ is always exactly divergence free. 

In particle methods, the particle positions $\vv{x}_i\in\Omega$, $i=1,\dotsc,N$,
are updated by solving $\tfrac{\mathrm d}{{\mathrm d}t}{\vv{x}_i}(t)
= \vv{U}\bigl(t,\vv{x}_i(t)\bigr)$, $i=1,\dotsc,N$ using a time-stepping scheme.
It makes sense to use a \emph{volume-preserving} scheme for this problem. However,
most of these schemes require a stream function $\vv{A}$ and not the velocity
$\vv{U}$ as input,\cite[Chapter~VI.9]{hairer2006} and thus arises the desire to
have a stream function of maximum regularity at hand. In this work we describe
how to compute stream functions that are at least $\HOneVec$-regular---even on
non-smooth domains---thereby improving on the regularity of approximations
currently available in related literature.

\subsection{Summary of Results}
Our results can be summarised as follows. 
\begin{enumerate}
	\item\textbf{Existence of Velocity Fields.}\label{item:existence}~%
	(\Cref{thm:existence of solutions})
	Problem~\eqref{eqn:velocity-recovery} has a solution $\vv{U}\in\vv{L}^2(\Omega)$
	if and only if $\vv{F}\in\HOneDualVec$ and $\langle\vv{F},\vv{V}\rangle = 0$ for
	all $\vv{V}\in\HOneZeroVec$ with $\curl\vv{V}=\gv{0}$. In \Cref{lem:alternative-conditions},
	we discuss equivalent alternative formulations of the latter condition.
	
	\item\textbf{Existence of Stream Functions.}\label{item:streamfct}~%
	(\Cref{thm:existence of stream function})
	Let the velocity $\vv{U}\in\LebVec$ solve problem~\eqref{eqn:velocity-recovery}.
	Then, $\vv{U}$ can be written in terms of a stream function $\vv{A}\in\HOneVec$
    as $\vv{U}=\curl\vv{A}$	if and only if $\vv{U}$ fulfils $\intdx{\Gamma_i}{}{
    \vv{U}\cdot\vv{n}}{S}=0$ on each connected component $\Gamma_i$ of the boundary
    $\Gamma\coloneqq \partial\Omega$.
	
	\item\label{uniqueness item}\textbf{Uniqueness.} (\Cref{thm:uniqueness of velocity field,%
	thm:uniqueness of stream function}) If $\Omega$ is \enquote{handle-free}, the
    solution $\vv{U}\in\vv{L}^2(\Omega)$ of problem	\eqref{eqn:velocity-recovery}
    can be made unique by prescribing its normal trace $\vv{U}\cdot\vv{n}\in
    H^{-\frac{1}{2}}(\Gamma)$. Moreover, if the prescribed boundary data fulfils
    $\intdx{\Gamma_i}{}{\vv{U}\cdot\vv{n}}{S}=0$ on each connected component
    $\Gamma_i\subset\Gamma$ of the boundary, there exist conditions that
    \emph{uniquely} determine a stream function $\vv{A}\in\vv{H}^1(\Omega)$ such that
	$\vv{U}=\curl\vv{A}$.
	
	\item\textbf{Construction of Solutions.} (\Cref{sec:construction}) The main
    novelty	of this work lies in the explicit construction of $\vv{H}^1(\Omega)$-regular
    stream functions $\vv{A}$ directly from the vorticity $\vv{F}$. Given a vorticity
     $\vv{F}\in\HOneDualVec$ fulfilling the conditions of \cref{item:existence} and
    boundary data $\vv{U}\cdot\vv{n}\in H^{-\frac{1}{2}}(\Gamma)$ fulfilling the
    conditions of \cref{item:streamfct}, this construction will yield a stream
    function $\vv{A}\in\vv{H}^1(\Omega)$ such that $\vv{U}=\curl\vv{A}$ solves
    problem~\eqref{eqn:velocity-recovery}. If the domain is handle-free, the obtained
    solution will be the uniquely defined stream function $\vv{A}\in\HOneVec$ from
    \cref{uniqueness item}. Numerical methods will be described in \Cref{sec:numerics}.
	
	\item\textbf{Well-posedness.} (\Cref{thm:well-posedness}) From the structure of the
	construction one can directly infer its well-posedness. The vector-fields
	$\vv{U}$ and $\vv{A}$ continuously depend on the given data $\vv{F}\in%
	\HOneDualVec$ and $\vv{U}\cdot\vv{n}\in H^{-\frac{1}{2}}(\Gamma)$.
	
	\item\textbf{Regularity.} (\Cref{thm:regularity}) If in addition to the above assumptions the given
	data fulfils $\vv{F}\in\LebVec$ and $\vv{U}\cdot\vv{n}\in L^2(\Gamma)$, then
	$\vv{U}\in\vv{H}^{\frac{1}{2}}(\Omega)$ and $\vv{A}\in\vv{H}^{\frac{3}{2}}(\Omega)$.
\end{enumerate}
The results concerning
existence and uniqueness of velocity fields $\vv{U}$ follow from classical
functional analytic arguments and are well-known, but covered for completeness.
Existence of stream functions $\vv{A}\in\vv{H}^1(\Omega)$ is due to
Girault and Raviart\cite[Theorem~3.4]{raviart1986}. Their work, however, left
unclear how to compute such a potential.

\subsection{Problematic Approaches}\label{sec:problematic-approaches}
A naive approach to the div-curl problem~\eqref{eqn:velocity-recovery} relies on the observation that
\begin{equation}
\Laplacevec\vv{U} = \curl(\underbrace{\curl\vv{U}}_{=\vv{F}})-\gradp(\underbrace{\DIV\vv{U}}_{=0})
                  = \curl\vv{F}.
\end{equation}
Based on this vector-calculus identity, it is tempting to solve three \emph{decoupled} scalar
Poisson problems $\Laplace U_i=(\curl\vv{F})_i$, $i=1,2,3$, for the components of $\vv{U}$, say
by prescribing the value of each one on the boundary. However, this approach is problematic: it
is our aim to \emph{integrate $\vv{F}$}, but instead this strategy asks that we \emph{differentiate}
first. Therefore, it needlessly requires to impose more regularity on the right-hand side. Moreover,
there is no guarantee that its solution is divergence-free. Finally, since the tangential components
of $\vv{U}$ allow us to compute $(\curl\vv{U})\cdot\vv{n}$ on the boundary, the boundary data must
fulfil the compatibility condition $(\curl\vv{U})\cdot\vv{n}=\vv{F}\cdot\vv{n}$. We will later see
that the solutions of problem~\eqref{eqn:velocity-recovery} are \emph{usually not} $\HOneVec$-regular,
and thus the classic existence and uniqueness results in $\HOne$ for the scalar Poisson problems
$\Laplace U_i=(\curl\vv{F})_i$ are not applicable either.

Another straightforward approach assumes that $\vv{F}\in\LebVec$. One may then
extend $\vv{F}$ by zero to the whole space, yielding $\widetilde{\vv{F}}\in
\vv{L}^2(\wholespace)$, and apply the Biot--Savart law to this extension. The normal
trace $\vv{U}\cdot\vv{n}$ on $\Gamma$ can then be prescribed by adding a suitable
\enquote{potential flow}. The main caveat of this strategy is that unless $\vv{F}\cdot\vv{n}=0$
on the boundary, the zero extension $\widetilde{\vv{F}}$ will \emph{not} be divergence-%
free, and in this case the Biot--Savart law fails to yield the correct result. We
will later see that this approach can in fact be fixed by introducing a suitable
correction on the boundary.

\subsection{Our Results in Context}
Clearly, for a given velocity field $\vv{U}$, the condition $\vv{U}=\curl\vv{A}$
alone does not uniquely determine~$\vv{A}$: because of the vector calculus
identity $\curl\circ\,(\grad)\equiv\gv{0}$, any gradient may be added to $\vv{A}$
without changing its curl. It is thus natural to enforce the so-called Coulomb
gauge condition $\DIV\vv{A}=0$, but this alone still does not ensure uniqueness.
For a given $\vv{F}$, we are then in fact facing two systems:
\begin{equation}
\begin{array}{rcl c rcl}
\curl\vv{A} &=& \vv{U},  &\multirow{2}{2cm}{\centering\text{and}} &  \curl\vv{U} &=& \vv{F}, \\
\DIV \vv{A} &=& 0,       &                                        &  \DIV \vv{U} &=& 0.
\end{array}
\end{equation}
These systems differ in the involved spaces and boundary conditions. For the
$\vv{U}$-system we would like to prescribe $\vv{U}\cdot\vv{n}$ on $\pd\Omega$,
while for the $\vv{A}$-system we actually do not care which boundary conditions
are prescribed, as long as they ensure that the solution is unique, and---%
hopefully---as regular as possible.

Many results in the literature are concerned with only one of these systems,
an overview of some results is given in~\Cref{tab:literature-review}{\kern-5pt}.
For example, the famous work of Amrouche et al.\cite{amrouche1998} is concerned
with the $\vv{A}$-system and $\vv{U}\in\LebVec$. They propose tangential
or normal boundary conditions for $\vv{A}$ and numerical methods to approximate
the resulting stream functions. However, even for perfectly smooth velocity
fields $\vv{U}$, the resulting potentials will usually only have Sobolev
regularity $\vv{H}^{\frac{1}{2}}(\Omega)$, unless the domain is assumed to be
more regular than just Lipschitz. In particular, functions from $\Hcurl\cap\HOdiv$
may develop quite strong singularities near corners of the domain, which makes
it difficult to approximate them efficiently.\cite[Figure~1.3]{arnold2018} Note
that these singularities can only occur in non-smooth domains: for $C^{1,1}$-%
domains one has $\vv{A}_{\mathrm{T}},\vv{A}_{\mathrm{N}}\in\vv{H}^1(\Omega)$, so
in this case there is no need to look for more regular potentials. Almost all of
the literature concerning numerics for the $\vv{A}$-system considers either
$\vv{A}_{\mathrm{T}}$ or $\vv{A}_{\mathrm{N}}$.\cite[Chapter~6.1]{alonso2010}
Many authors also consider more general problems involving inhomogeneous
material coefficients,\cite{auchmuty2005} or the $L^p$-setting,\cite{kozono2009,%
amrouche2013} which on the other hand again often requires higher regularity
assumptions on the boundary.

\begin{table}
\centering
\begin{tabular*}{\textwidth}{lclll}
\toprule
Reference                              & Regularity of $\Omega$ & Input                    & Output                                               & Remarks \\\midrule
Amrouche et al.\cite{amrouche1998}     & Lipschitz              & $\vv{U}\in\LebVec$       & $\vv{A}_{\mathrm{T}}\in\vv{H}^{\frac{1}{2}}(\Omega)$ & tangential potential\\
Amrouche et al.\cite{amrouche1998}     & Lipschitz              & $\vv{U}\in\LebVec$       & $\vv{A}_{\mathrm{N}}\in\vv{H}^{\frac{1}{2}}(\Omega)$ & normal potential, only if $\vv{U}\cdot\vv{n}=0$\\
Bramble and Pasciak\cite{bramble2004}  & Lipschitz              & $\vv{F}\in\HOneDualVec$  & $\vv{U}\in\LebVec$                                   & \\
Alonso and Valli\cite{alonso1996}      & $C^{1,1}$              & $\vv{F}\in\LebVec$       & $\vv{U}\in\LebVec$                                   & \\\midrule
This work                              & Lipschitz              & $\vv{F}\in\HOneDualVec$  & $\vv{A}\in\HOneVec$                                  & solves both systems simultaneously\\\bottomrule
\end{tabular*}
\caption{\label{tab:literature-review}Approaches found in the literature for
solving related div-curl problems and the \emph{component-wise} Sobolev
regularity of the input and output data. No other approach known to the authors
achieves $\HOneVec$-regularity of $\vv{A}$ in non-smooth Lipschitz domains.}
\end{table}

It has long been known that stream functions of regularity $\vv{H}^1(\Omega)$ do
exist, but so far conditions that uniquely characterise them have not been given.
We are unaware of any previous approaches that allow us to efficiently compute
$\vv{A}\in\vv{H}^1(\Omega)$ for the case of general Lipschitz domains. Our work
aims to close this gap. It proposes natural conditions which \emph{uniquely}
determine a vector potential $\vv{A}_1\in\HOneVec$ \emph{without} explicitly
involving boundary values of $\vv{A}_1$. We believe it is because previous
approaches do prescribe boundary conditions like $\vv{A}_{\mathrm{T}}\cdot\vv{n}
=0$ or $\vv{A}_{\mathrm{N}}\times\vv{n}=\gv{0}$ that they yield less regular
stream functions.
 
Our algorithm utilises the Newton operator: the bulk of the work lies in the
explicit computation of a volume integral. Two companion \emph{scalar} elliptic
equations must also be solved on the boundary of the domain, but these are easily
tackled using standard boundary element methods. While this construction solves
both the $\vv{U}$- and $\vv{A}$-systems simultaneously, we first discuss existence
and uniqueness of solutions for each of them individually.

\section{Basic Definitions and Notation}

\subsection{Spaces Defined on Volumes}\label{sec:Spaces defined on Volumes}
We denote by $\testfcts\coloneqq C^{\infty}_0(\Omega)$ the space of smooth
compactly supported functions in $\Omega$, and write $\distris$  for the space of
distributions. Their vector-valued analogues $\testfctsvec\coloneqq
\bigl(C^{\infty}_0(\Omega)\bigr)^3$ and $\distrisvec$ are distinguished by a
\textbf{bold} font. In the whole space $\wholespace$, we will make use of the
space of smooth functions $\wholesmoothfcts\coloneqq C^\infty(\wholespace)$ and
its dual $\compdistris$---the space of compactly supported distributions. Their
vector-valued analogues will be denoted by $\wholesmoothfctsvec$ and
$\compdistrisvec$, respectively.

We write $L^2(\Omega)$ and $\mathbf{L}^2(\Omega)$ for the Hilbert spaces of
square integrable scalar and vector-valued functions defined over $\Omega$.
$H^s(\Omega)$ and $\mathbf{H}^s(\Omega)$, $s>0$, refer to the corresponding
Sobolev spaces. The spaces $H_0^s(\Omega)$ and $\vv{H}_0^s(\Omega)$ are defined
as the closures of $\testfcts$ and $\testfctsvec$ in $H^s(\Omega)$ and
$\vv{H}^s(\Omega)$, respectively. For $s<0$ we set $H^{s}(\Omega)\coloneqq 
H^{-s}_0(\Omega)'$ and $\vv{H}^{s}(\Omega)\coloneqq \vv{H}^{-s}_0(\Omega)'$.
We will always identify $L^2(\Omega)$ and $\LebVec$ with their duals, i.\,e.,
$L^2(\Omega)'= L^2(\Omega)$ and $\LebVec' = \LebVec$. The Hilbert spaces
$\Hdiv\coloneqq \{\vv{U}\in\vv{L}^2(\Omega)\,\vert\,\DIV\vv{U}\in L^2(\Omega)\}$
and
$\Hcurl \coloneqq  \{\mathbf{U}\in\LebVec\,\vert\,\curl\mathbf{U}\in\LebVec\}$
are equipped with the obvious graph norms. The related \enquote{homogeneous
spaces} are defined as $\HOcurl\coloneqq \overline{\testfctsvec}^{\Hcurl}$ and
$\HOdiv\coloneqq\overline{\testfctsvec}^{\Hdiv}$. Accordingly, all differential
operators are to be understood in the distributional sense. We refer to Amrouche
et al. for a detailed exposition of the regularity and compactness properties of
these spaces.\cite[Sections~2.2 and~2.3]{amrouche1998} Evidently, these
definitions can also be used with $\Omega$ replaced by $\wholespace$ and vice-%
versa.

\subsection{Trace Spaces}
We refer to McLean\cite[Chapter~3]{mclean2000}, Sauter and Schwab\cite[Chapter~2]{sauter2011}
and Buffa et al.\cite {buffa2002} for theory concerning extension of the traces
\begin{equation}
\gamma V       \coloneqq V|_\Gamma,                           \quad 
\nu\vv{u}      \coloneqq \vv{u}|_\Gamma\cdot\vv{n},           \quad
\gv{\tau}\vv{u}\coloneqq \vv{u}|_\Gamma - (\vv{u}|_\Gamma\cdot\vv{n})\vv{n}
                = \vv{n}\times(\vv{u}|_\Gamma\times\vv{n}),   \quad
\text{and}                                                    \quad
\gv{\rho}\vv{u}\coloneqq \vv{u}|_\Gamma\times\vv{n},
\end{equation}
to continuous and surjective mappings
\begin{equation}
\begin{aligned}
\gamma&:    &  \HOne &\to H^{\frac{1}{2}}(\Gamma), &  \ker(\gamma)   &=\HOneZero,\\
\nu&:       &  \Hdiv &\to\SobTraceDual,            &  \ker(\nu)      &=\HOdiv,   \\
\gv{\tau}&: &  \Hcurl&\to\HTMinusCurl,             &  \ker(\gv{\tau})&=\HOcurl,  \\
\gv{\rho}&: &  \Hcurl&\to\HRMinusDiv,              &  \ker(\gv{\rho})&=\HOcurl,
\end{aligned}
\end{equation}
having continuous right-inverses (so-called \textit{lifting maps}). The fact that
the operators $\gv{\rho}$ and $\gv{\tau}$ are surjective is one of the main results
of Buffa et al.\cite{buffa2002} It allows for the derivation of Hodge decompositions,
which we will also use later on. The surface differential operators:
\begin{equation}
\tracediv:  \HRMinusDiv\to H^{-\frac{1}{2}}(\Gamma)
\qquad\text{and}\qquad
\tracecurl: \HTMinusCurl \to H^{-\frac{1}{2}}(\Gamma),
\end{equation}
used in those definitions are defined by duality for all $v\in H^{\frac{1}{2}}
(\Omega)$ by $\langle\tracediv\vv{u},v\rangle_\Gamma\coloneqq\langle\vv{u},
\tracegrad v\rangle_\Gamma$ and $\langle\tracecurl\vv{u},v\rangle_\Gamma\coloneqq
\langle\vv{u},\tracecurlvec v\rangle_\Gamma$. Here $\tracegrad$ and
$\tracecurlvec$ are suitable extensions of the ordinary trace gradient and trace
curl to operators with mappings $\tracegrad:H^{\frac{1}{2}}(\Gamma)\to
\HTMinusCurl$ and $\tracecurlvec:H^{\frac{1}{2}}(\Gamma)\to\HRMinusDiv$. These
operators also give rise to a duality between $\HRMinusDiv$ and $\HTMinusCurl$.

Finally, we will also make use of the Laplace--Beltrami operator $\tracelaplace:=
\tracediv\circ\,(\tracegrad)\equiv\tracecurl\circ\tracecurlvec$. This operator
is known to be coercive on $H^1(\Gamma)/\reals$, the space of $H^1(\Gamma)$-%
regular traces with zero average on each $\Gamma_i$, $i=0,\dotsc,\beta_2$.

\subsection{Geometry of the Domain}
Throughout this article, we suppose that the Lipschitz domain $\Omega\subset\wholespace$
of interest is bounded and connected, and we will sometimes refer to the Betti
numbers $\beta_1$ and $\beta_2$. These numbers are related to the topological properties
of the domain, see \Cref{fig:domain}{\kern-5pt}. $\beta_1$ is the genus of the domain,
in other words the number of \enquote{handles}. $\beta_2$ is the number of `holes'
$\Theta_i$, $i=1,\ldots,\beta_2$ in the domain. We define the \emph{exterior domain}
$\Theta_0$ via:
\begin{equation}
	\Theta_0\coloneqq\wholespace\setminus\overline{\Omega\cup\left(\bigcup_{i=1}^{\beta_2}\Theta_i\right)}.
\end{equation}
The domain's boundary thus always has $\beta_2+1$ connected components $\Gamma_i
\coloneqq\partial\Theta_i$, $i=0,\ldots,\beta_2$. A domain with $\beta_1=0$ is called
\emph{handle-free},\footnote{We avoid the term \enquote{simply-connected} as it
usually refers to \emph{homotopy} as opposed to \emph{homology}.}
\emph{hole-free} if $\beta_2=0$, and we say that the topology of $\Omega$ is
trivial or simple if $\beta_1=\beta_2=0$. We refer to Arnold and al. for
more details.\cite[Section~2]{arnold2006}

The geometric interpretation of the Betti numbers is best illustrated through an
example. In the domain depicted in \Cref{fig:domain}{\kern-5pt}, $\beta_2=3$: three
cube-like holes $\Theta_1$, $\Theta_2$, and $\Theta_3$ were cut out of the toroidal
volume. Their boundaries $\Gamma_1$, $\Gamma_2$, and $\Gamma_3$ are labelled in the
figure. Together with the exterior boundary $\Gamma_0$, the boundary $\Gamma\coloneqq
\partial\Omega=\Gamma_0\cup\Gamma_1\cup\Gamma_2\cup\Gamma_3$ thus has \emph{four} =
$\beta_2+1$ connected components.

On the one hand, the value of the second Betti number $\beta_2$ is relevant to
questions regarding the \emph{existence} results stated in \cref{item:existence}
and \cref{item:streamfct} of \Cref{sec: Introduction}. These existence theorems
will make use  of arbitrary but fixed functions $T_i\in C_0^\infty(\wholespace)$,
$i=0,\dotsc,\beta_2$, that act as indicators for the the connected components of
the boundary:
\begin{equation}\label{def: Ti}
	T_i =
	\begin{cases}
		1 & \text{in a neighbourhood of $\Gamma_i$,}\\
		0 & \text{in a neighbourhood of $\Gamma_j$, $j\in\lbrace 0,\ldots,\beta_2\rbrace\setminus\lbrace i\rbrace$}.
	\end{cases}
\end{equation}

On the other hand, the value of $\beta_1$ is crucial to the \emph{uniqueness} results
of \cref{uniqueness item}. For simplicity, we will restrict our attention to handle-free
domains ($\beta_1=0$), that is domains for which every loop inside $\Omega$ is the boundary
of a surface within $\Omega$. The domain in \Cref{fig:domain}{\kern-5pt} is \emph{not}
handle-free,  as the red loop is a representative of the equivalence class of
non-bounding cycles. In that example, $\beta_1=1$. Nevertheless, we will make some
remarks on what changes in the following results need to be anticipated in order
to recover uniqueness of solutions when $\beta_1>0$.

\begin{figure}
\centering
\includegraphics[width=0.4\textwidth]{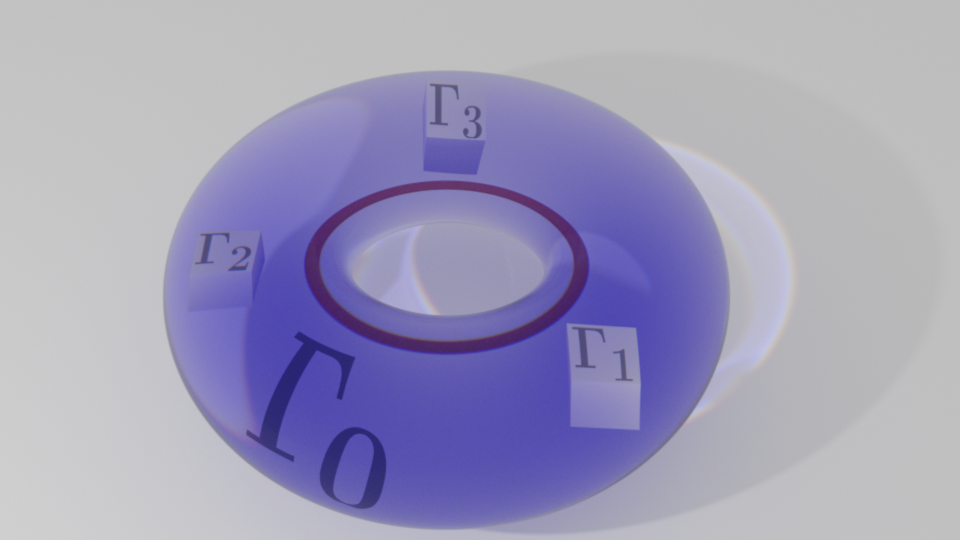}
\caption{\label{fig:domain}A ring-shaped domain with three cubical holes is
an example of a domain having non-trivial topology.\cite[Section~3]{amrouche1998}
There is one \enquote{handle} through it: $\beta_1=1$. The red line is a representative
of the equivalence class of non-bounding cycles. The three cubical inclusions
(\enquote{holes}) are not part of the domain: $\beta_2=3$. The boundary
$\Gamma$ has four connected components: $\Gamma_1,\Gamma_2,\Gamma_3$, and the
domain's exterior boundary $\Gamma_0$.}
\end{figure}

\subsection{Laplace and Newton Operator}
The scalar Laplacian $\Laplace\coloneqq\DIV\circ\,(\grad)$ is central to
potential theory, and we assume that the reader is well aware of the classical
existence and uniqueness results for boundary value problems related to this
operator. Its vector-valued analogue is defined component-wise:
$\Laplacevec\vv{V} \coloneqq (\Laplace V_1,\Laplace V_2,\Laplace V_3)^\top$ for
all $\vv{V}\in\distrisvec$. This operator is also known as Hodge--Laplacian and
can equivalently be written as:
\begin{equation}
\Laplacevec\vv{V} = \curl\curl\vv{V}-\gradp\DIV\vv{V},\qquad\forall\vv{V}\in\distrisvec.
\end{equation}

Let us denote by $G(\vv{x})\coloneqq(4\pi|\vv{x}|)^{-1}$ the fundamental solution
of the Laplacian $\Laplace$. The Newton operator is defined on the space of
compactly supported distributions $\compdistris$ via convolution as
$
\Newton: \compdistris\to\wholedistris$, $U\mapsto G\star U
$.
In other words:
\begin{equation}
\forall U\in\wholetestfcts:\quad
\bigl(\Newton U\bigr)(\vv{x}) \coloneqq
\frac{1}{4\pi}\intdx{\wholespace}{}{\frac{U(\vv{y})}{|\vv{x}-\vv{y}|}}{\vv{y}}
\qquad\in\wholesmoothfcts,
\end{equation}
while $\langle\Newton U, V\rangle \coloneqq \langle U,\Newton V\rangle$ for all
$U\in\compdistris$ and $V\in\wholetestfcts$. For a given $U\in\compdistris$, 
the distribution $\Newton U$ is called the \emph{Newton potential} of $U$.

This operator is an inverse for the Laplacian:\cite[Equations~(4.4.2)
and~(4.4.3)]{hoermander1990}
\begin{equation}\label{eq: Neuton is Laplacian inverse}
	\forall U\in\compdistris:\qquad \Laplace\Newton U = \Newton(\Laplace U) = U.
\end{equation}
Moreover, because it is an operator of convolutional type, it commutes with
differentiation.\cite[Equation~(4.2.5)]{hoermander1990} Application of this
operator always increases the Sobolev regularity of a distribution $U\in H^{s}%
(\wholespace)\cap\compdistris$ by two, i.\,e., $\Newton$ has the following
mapping property and is continuous:\cite[Theorem~3.1.2]{sauter2011}
\begin{equation}
	\Newton: H^s(\wholespace)\cap\compdistris\to H^{s+2}_{\mathrm{loc}}(\wholespace)
	\qquad s\in\mathbb{R}.
\end{equation}

The Newton potential $\Newton U$ of a compactly supported distribution
$U\in\compdistris$ is regular outside $\supp U$ and decays to zero at
infinity:\cite[Chapter~II, \S 3.1, Proposition~2]{dautray1990-i}
\begin{equation}
	\forall U\in\compdistris:\ 
	\bigl(\Newton U\bigr)(\vv{x})=\bigO{|\vv{x}|^{-1}}\qquad|\vv{x}|\to\infty.
\end{equation}
Even more importantly, the following result characterises the Newton potential.%
\cite[Chapter~II, \S 3.1, Proposition~3]{dautray1990-i}
\begin{lemma}\label{lem:newton-characterisation}
	Let $U\in\wholedistris$ and $F\in\compdistris$. Then $U$ is the Newton
	potential of $F$, that is $U=\Newton F$, if and only if:
	\begin{equation}
		\left\lbrace
		\begin{aligned}
			\Laplace U         &=   F &   &\text{on $\wholespace$,}\\
			U(\vv{x}) &\to 0 &   &\text{as $|\vv{x}|\to\infty$.}
		\end{aligned}
		\right.
	\end{equation}
\end{lemma}
This characterisation allows for the derivation of representation formulæ for
solutions of the Laplace equation on bounded domains, leading to boundary integral 
equations. This will appear at the end of this section. All of these results
analogously hold for the vector-valued Newton operator $\NewtonVec$, which is
defined component-wise and distinguished by a bold-face font.

\subsection{Decompositions of Helmholtz--Hodge Type}
Decomposition theorems will play a central role in our analysis,
so we collect the most important results here.
\begin{lemma}[Helmholtz Decomposition]\label{lem:helmholtz-decomposition}
Every compactly supported distribution $\vv{U}\in\compdistrisvec$ can be
decomposed into a divergence-free and a curl-free part:
\begin{equation}\label{eqn:helmholtz-decomposition}
		\vv{U} = \curl\vv{A} - \gradp P,
\end{equation}
where $\vv{A}\coloneqq\NewtonVec\curl\vv{U}\in\wholedistrisvec$ and
	$P\coloneqq\Newton\DIV\vv{U}\in\wholedistris$.
\end{lemma}
\begin{proof}
\begin{math}
\vv{U} = \Laplacevec\NewtonVec\vv{U}
       = \curl(\curl\NewtonVec\vv{U}) -\gradp(\DIV\NewtonVec\vv{U})
       = \curl(\NewtonVec\curl\vv{U}) -\gradp(\Newton\DIV\vv{U}).
\end{math}
\end{proof}

Note that in this decomposition neither $\curl\vv{A}$ nor $\grad p$ are necessarily
compactly supported. This makes the following result useful, for which we refer
to the works of Bramble and Pasciak\cite[Proposition~3.2]{bramble2004},
Pasciak and Zhao\cite[Lemma~2.2]{pasciak2002}, as well as Hiptmair and Pechstein%
\cite[Remark~3]{hiptmair2019}.
\begin{lemma}\label{lem:helmholtz-HOcurl}
Let $\Omega\subset\wholespace$ be a bounded Lipschitz domain. Then any $\vv{U}%
\in\HOcurl$ can be decomposed as:
\begin{equation}
\vv{U} = \vv{W}-\gradp P,
\end{equation}
for some $P\in\HOne$, where $\vv{W}\in\HOneZeroVec$ satisfies:
$\Vert\vv{W}\Vert_{\HOneVec}\leq C\Vert\curl\vv{U}\Vert_{\LebVec}.$
\end{lemma}

Finally, the tangential trace spaces also allow similar decompositions.
We only require the result for $\HRMinusDiv$.
\begin{lemma}[Hodge Decomposition~{\cite[Theorem~5.5]{buffa2002}}]\label{lem:trace-hodge-decomposition}
Let $\Gamma=\partial\Omega$ be the boundary of a handle-free, bounded Lipschitz
domain $\Omega\subset\wholespace$. Then any $\vv{s}\in\HRMinusDiv$ can be
uniquely decomposed as:
\begin{equation}
\vv{s} = \tracecurlvec p - \tracegradp q,
\end{equation}
where $p\in H^{\frac{1}{2}}(\Gamma)/\reals$ and $q\in H^1(\Gamma)/\reals$ with
$\tracelaplace q\in H^{-\frac{1}{2}}(\Gamma)$ are uniquely determined up to a
constant on each connected component $\Gamma_i$, $i=0,\dotsc,\beta_2$ of the
boundary.
\end{lemma}

\subsection{Trace Jumps and a Representation Formula}
The trace operators introduced in the previous sections were all defined with
respect to the domain $\Omega$. One can instead also consider the corresponding
traces with respect to the complementary domain
$\Omega^C\coloneqq\wholespace\setminus\overline{\Omega}$. These
\emph{one-sided} traces exist whenever the restriction of a
vector field $\vv{U}\in L^2_{\mathrm{loc}}(\wholespace)$ to the domains
$\Omega$ and $\Omega^C$ is sufficiently smooth. If $\vv{U}$ is smooth
\emph{across $\Gamma$}, then the one-sided traces coincide. Otherwise the
difference of these traces is denoted by the jump operator $\llbracket\cdot\rrbracket$.
For example, for $\gv{\rho}$ one  writes: $\llbracket\gv{\rho}\vv{U}\rrbracket
\coloneqq \gv{\rho}\vv{U}|_{\Omega} - \gv{\rho}^C \vv{U}|_{\Omega^C} $. The importance
of the jump operator lies in the following representation formula.
\begin{lemma}[{Representation Formula \cite[Sec. 4.2]{claeys2019}}]\label{lem:representation-formula}
Let $\vv{U}\in\HlocVec$ fulfil:
\begin{equation}
\left\lbrace
\begin{aligned}
\Laplacevec\vv{U} &=  \gv{0}   & &\text{in $\wholespace\setminus\Gamma$},\\
\DIV\vv{U}        &=  0        & &\text{in $\wholespace$},\\
\vv{U}(\vv{x})    &\to \gv{0}  & &\text{as $|\vv{x}|\to\infty$}.
\end{aligned}
\right.
\end{equation}
Then $\vv{U}=\NewtonVec(\Laplacevec\vv{U})$ and $\Laplacevec\vv{U}=%
\gv{\tau}'\llbracket\gv{\rho}\curl{\vv{U}}\rrbracket$.
\end{lemma}
\begin{proof}
The fact that $\vv{U}=\NewtonVec(\Laplacevec\vv{U})$ directly follows
from \Cref{lem:newton-characterisation}. Because $\DIV\vv{U}=0$ globally, we
have $\Laplacevec\vv{U} = \curl(\curl\vv{U})$ and
$\curl(\curl{\vv{U}})|_{\wholespace\setminus\Gamma}=\gv{0}$. We thus obtain
for all $\vv{V}\in\wholetestfctsvec$:
\begin{equation}
\langle\Laplacevec\vv{U},\vv{V}\rangle =
\langle\curl\vv{U},\curl\vv{V}\rangle =
\int_{\Omega}\curl\vv{U}\cdot\,\curl\vv{V}\,{\mathrm d}\vv{x} +
\int_{\wholespace\setminus\overline{\Omega}}\curl\vv{U}\cdot\,\curl\vv{V}\,{\mathrm d}\vv{x},
\end{equation}
where we used that $\curl\vv{U}\in\vv{L}_{\text{loc}}^2(\wholespace)$ since
$\vv{U}\in\HlocVec$. Now, by definition of the rotated tangential trace:
\begin{equation}
\langle\gv{\rho}\curl\vv{U},\gv{\tau}\vv{V}\rangle_\Gamma =
\int_{\Omega}\curl\vv{U}\cdot\,\curl\vv{V} -
\underbrace{\curl(\curl\vv{U})}_{=\gv{0}}\cdot\vv{V}\,{\mathrm d}\vv{x} =
\int_{\Omega}\curl\vv{U}\cdot\,\curl\vv{V}\,{\mathrm d}\vv{x}.
\end{equation}
Applying the same methodology to the integral over $\wholespace\setminus
\overline{\Omega}$ and using the definition of $\gv{\rho}^C$ yields the
desired result, because the fact that $\vv{V}$ is smooth across $\Gamma$
guarantees that $\llbracket\gv{\tau}\vv{V}\rrbracket=\gv{0}$.
\end{proof}

\section{Velocity Fields}
In this section we prove the existence of velocity fields solving
\eqref{eqn:velocity-recovery} as claimed in \Cref{item:existence}. The abstract
integrability condition is reformulated in \Cref{lem:alternative-conditions}. The
uniqueness result for velocity fields presented in \Cref{uniqueness item} is also
covered.

\subsection{Existence of Velocity Fields}
\begin{theorem}\label{thm:existence of solutions}
Suppose that $\Omega\subset\wholespace$ is a bounded Lipschitz domain. The div-%
curl system
\begin{equation}\label{eqn:div-curl-system}
\left\lbrace
\begin{aligned}
\curl \mathbf{U} &= \vv{F}\\
\DIV  \mathbf{U} &= 0
\end{aligned}
\right.\qquad\text{in }\Omega
\end{equation}
has a solution $\vv{U}\in\LebVec$ if and only if $\vv{F}$ lies in the space
$\HOneDualVec$ and fulfils the following integrability condition:
\begin{equation}\label{eqn:integrability condition}
\langle\vv{F},\vv{V}\rangle = 0 \qquad\forall\vv{V}\in\ker\curl\big|_{\HOneZeroVec}.
\end{equation}
\end{theorem}
\begin{remark}
A more general result was proven by Bramble and Pasciak,\cite[Theorem~4.1]{bramble2004}
using entirely different techniques. We follow another route that sheds new light
on the classical conditions imposed on $\vv{F}$ and that helps in the construction
of vector potentials later.
\end{remark}
\begin{remark}
The condition $\vv{F}\in\HOneDualVec$ might seem unnatural, because for an
arbitrary vector-field $\vv{U}\in\LebVec$ it holds that
\begin{equation}\label{eqn:curlhatbound}
\forall\vv{V}\in\testfctsvec:\quad
\langle\curl\vv{U},\vv{V}\rangle =
\intdx{\Omega}{}{\vv{U}\cdot\curl\vv{V}}{\vv{x}}\leq
\Vert\vv{U}\Vert_{\vv{L}^2(\Omega)}\Vert\vv{V}\Vert_{\vv{H}(\curl;\Omega)}.
\end{equation}
The distribution $\curl\vv{U}\in\distrisvec$ thus admits a unique continuous
extension to $\overline{\testfctsvec}^{\Hcurl}=\HOcurl$ and the associated
operator 
\begin{equation}\label{eq: curl restricted to L2}
\curl\big\vert_{\vv{L}^2(\Omega)}:\LebVec\to\HOcurldual
\end{equation}
is continuous. On the one hand, any solution $\vv{U}\in\LebVec$ of
\eqref{eqn:div-curl-system} must therefore also necessarily fulfil
$\vv{F}=\curl\vv{U}\in\HOcurldual$. On the other hand, we have the
\emph{proper} inclusion $\HOcurldual\subsetneq\HOneDualVec$. The
following result resolves this issue.
\end{remark}
\begin{lemma}\label{lmm:identification of spaces}
The two spaces:
\begin{alignat*}{3}
\biggl\lbrace\vv{F} &\in \HOcurldual  &\bigg|\ \langle\vv{F},\vv{V}\rangle &= 0
&\quad &\forall\vv{V}\in\ker\curl\big|_{\HOcurl}  \biggr\rbrace\\
\shortintertext{and}
\biggl\lbrace\vv{F} &\in \HOneDualVec &\bigg|\ \langle\vv{F},\vv{V}\rangle &= 0
&\quad &\forall\vv{V}\in\ker\curl\big|_{\HOneZeroVec} \biggr\rbrace
\end{alignat*}
coincide with equivalent norms: $\Vert\vv{F}\Vert_{\HOneDualVec} \leq
\Vert\vv{F}\Vert_{\HOcurldual} \leq C\Vert\vv{F}\Vert_{\HOneDualVec}$.
\end{lemma}
\begin{proof}
\enquote{$\subset$} The first inclusion is trivial because of the continuous
embedding $\HOcurldual\hookrightarrow\HOneDualVec$. We thus also immediately
obtain the first inequality $\Vert\vv{F}\Vert_{\HOneDualVec}\leq\Vert\vv{F}
\Vert_{\HOcurldual}$.

\enquote{$\supset$} Let $\vv{F}\in\HOneDualVec$ be as above, and let $\vv{V}\in
\testfctsvec$ be arbitrary. We use the decomposition from \Cref{lem:helmholtz-HOcurl}
and write $\vv{V}=\vv{W}-\gradp P$ for some $P\in\HOne$ and $\vv{W}\in\HOneZeroVec$
that satisfies $\Vert\vv{W}\Vert_{\HOneVec}\leq C\Vert\curl\vv{V}\Vert_{\LebVec}$.
As $\vv{W}\in\HOneZeroVec$, we necessarily also have $\grad P\in\HOneZeroVec$ and
we may write:
\begin{equation}
\langle\vv{F},\vv{V}\rangle = \langle\vv{F},\vv{W}\rangle -
\underbrace{\langle\vv{F},\gradp P\rangle}_{=0} \leq
\Vert\vv{F}\Vert_{\HOneDualVec} \Vert\vv{W}\Vert_{\HOneVec}
\leq C\Vert\vv{F}\Vert_{\HOneDualVec} \Vert\vv{V}\Vert_{\Hcurl}.
\end{equation}
The distribution $\vv{F}\in\distrisvec$ thus admits a unique
continuous extension to $\overline{\testfctsvec}^{\Hcurl} =
\HOcurl$ and thus $\vv{F}\in\HOcurldual$ with
$\Vert\vv{F}\Vert_{\HOcurldual}\leq C\Vert\vv{F}\Vert_{\HOneDualVec}$.
\end{proof}

\begin{lemma}\label{lmm:plain-existence}
Suppose that $\Omega\subset\wholespace$ is a bounded Lipschitz domain.The equation
\begin{equation}\label{eqn:plain-solution}
\curl\vv{W} = \vv{F}
\end{equation}
has a solution $\vv{W}\in\vv{L}^2(\Omega)$ if and only if $\vv{F}\in\HOneDualVec$
and $\vv{F}$ fulfils the integrability condition \eqref{eqn:integrability condition}.
\end{lemma}
\begin{proof}
The continuous operator
\begin{equation}\label{eq: curl Hcurl to L2}
\curl\big\vert_{\HOcurl}:\HOcurl\rightarrow \vv{L}^2(\Omega)
\end{equation}
has closed range.\cite[Box~3.1]{arnold2018} The curl operator is symmetric and
the dual of the mapping \eqref{eq: curl Hcurl to L2} is the operator $\curl\big%
\vert_{\vv{L}^2(\Omega)}$ given in \eqref{eq: curl restricted to L2}. Hence,
Banach's closed range theorem yields
\begin{equation}
\range\left(\curl\big\vert_{\vv{L}^2(\Omega)}\right) =
\left(\ker\curl\big\vert_{\HOcurl}\right)^0
\end{equation}
That is,
\begin{equation}
\range\left(\curl\big\vert_{\vv{L}^2(\Omega)}\right) =
\biggl\lbrace \vv{F}\in\HOcurldual\ \bigg|\ \langle\vv{F},\vv{V}\rangle = 0
\quad\forall\vv{V}\in\ker\curl\big|_{\HOcurl}\biggr\rbrace.
\end{equation}
Evidently, problem \eqref{eqn:plain-solution} has a solution if and only if
$\vv{F}\in\range\,\left(\curl\big\vert_{\vv{L}^2(\Omega)}\right)$. Thus,
together with~\Cref{lmm:identification of spaces}, the claim follows.
\end{proof}

\begin{proof}[Proof of \Cref{thm:existence of solutions}]
\Cref{lmm:plain-existence} guarantees the existence of a $\vv{W}\in%
\LebVec$ such that $\curl\vv{W}=\vv{F}$. This function does not
necessarily fulfil $\DIV\vv{W}=0$. But in this case we let $P\in
H^1_0(\Omega)$ denote the unique solution to the Poisson problem 
\begin{equation}
\Laplace P = \DIV\vv{W}\qquad\text{in $\Omega$},
\end{equation}
and note that  $\vv{U}\coloneqq\vv{W}+\gradp P$ solves the div-curl
system~\eqref{eqn:div-curl-system}.
\end{proof}

The integrability condition~\eqref{eqn:integrability condition} is most
natural for the chosen method of proof. However, it is hard to
verify in practice. For this reason, it is worthwhile considering equivalent
alternative conditions.

\begin{lemma}\label{lem:alternative-conditions}
Suppose that $\Omega\subset\wholespace$ is a bounded Lipschitz domain and let
$\vv{F}\in\HOneDualVec$. Together, the following conditions are equivalent to
the integrability condition \eqref{eqn:integrability condition}:
\begin{subequations}
\begin{alignat}{2}
\label{eqn:intcond1}
\DIV\vv{F}                                       &= 0  &\qquad   &\text{in }\Omega,\\
\langle\vv{F},\grad T_i\big\vert_{\Omega}\rangle &= 0  &\qquad   &i=1,\dotsc,\beta_2.
\label{eqn:intcond2}
\end{alignat}
\end{subequations}
If in particular $\vv{F}\in\vv{L}^2(\Omega)$ and $\DIV\vv{F}=0$,
condition~\eqref{eqn:intcond2} is equivalent to:
\begin{equation}\label{eqn:intcond3}
\int_{\Gamma_i}\vv{F}\cdot\vv{n}\,{\mathrm d}S = 0\qquad i=1,\dotsc,\beta_2.
\end{equation}
\end{lemma}
\begin{remark}
Notice that definition \eqref{def: Ti} guarantees that $\grad T_i\big\vert_{\Omega}\in
\testfctsvec$.
\end{remark}
\begin{remark}
Together \eqref{eqn:intcond1} and \eqref{eqn:intcond2} also imply that
$\langle\vv{F},\grad T_0|_\Omega\rangle = 0$ holds. If in particular $\beta_2=0$,
it suffices to demand $\DIV\vv{F}=0$.
\end{remark}
\begin{proof}
\enquote{$\Rightarrow$} Since $\curl\circ\,(\grad)\equiv\gv{0}$, 
conditions~\eqref{eqn:intcond1} and~\eqref{eqn:intcond2} are
immediately seen to be necessary from the definitions. 
	
\enquote{$\Leftarrow$} To see that they also are sufficient, let
$\vv{V}\in\ker\curl\big|_{\HOneZeroVec}$ be arbitrary. We may extend this
function by zero outside $\Omega$:
\begin{equation}
\widetilde{\vv{V}}:\wholespace\to\wholespace,\qquad \vv{x}\mapsto
\begin{cases}
\vv{V}(\vv{x}) & \vv{x}\in\Omega, \\
\gv{0}         & \text{else.}
\end{cases}
\end{equation}
Since $\gv{\rho}(\vv{V})=\gv{0}$, we have $\curl\widetilde{\vv{V}}=\gv{0}$
on all of $\wholespace$. Since its support is compact, we may use the Helmholtz
decomposition~\eqref{eqn:helmholtz-decomposition} to rewrite this extension in
terms of
\begin{equation}
\widetilde{\vv{V}} =  \curl\mathcalbf{N}\underbrace{\curl\widetilde{\vv{V}}}_{=\gv{0}}
- \nabla\underbrace{\mathcal{N}\DIV\widetilde{\vv{V}}}_{=:\widetilde{P}}
= -\nabla\widetilde{P}.
\end{equation}
	
The restriction $P\coloneqq\widetilde{P}\big\vert_{\Omega}$ belongs to $H^1(\Omega)$,
because $\grad P =\vv{V}\in\LebVec$. Moreover, we see from $\gv{\tau}(\vv{V}) =
\gv{0}$ that $P=C_i$ for some constant $C_i\in\mathbb{R}$ on each connected
component $\Gamma_i$ of the boundary, $i=0,1,\dotsc,\beta_2$. Because
$\widetilde{P}\to 0$ at infinity, $\grad\widetilde{P}=\gv{0}$ outside $\Omega$,
and $\widetilde{P}\in\Hloc$ we have $C_0=0$. From the decomposition
\begin{equation}
P = \underbrace{\left(P-\sum_{i=1}^{\beta_2}C_iT_i\big\vert_{\Omega}\right)}_{%
=:P_0\in H^1_0(\Omega)}
+\sum_{i=1}^{\beta_2}C_iT_i\big\vert_{\Omega},
\end{equation}
we obtain
\begin{equation}
\langle\vv{F},\vv{V}\rangle =  \langle\vv{F},\grad P\rangle =
\underbrace{\langle\vv{F},\grad P_0\rangle}_{=0,\ \eqref{eqn:intcond1}}
+\sum_{i=1}^{\beta_2}C_i\underbrace{\langle\vv{F},\grad T_i\big\vert_{\Omega}%
\rangle}_{=0,\ \eqref{eqn:intcond2}} = 0.
\end{equation}
Thus \eqref{eqn:integrability condition} is equivalent to the combination of
\eqref{eqn:intcond1} and \eqref{eqn:intcond2}.
	
Finally, the equivalence of \eqref{eqn:intcond2} and \eqref{eqn:intcond3}
directly follows from the definition of the normal trace: if $\vv{F}\in\LebVec$
and $\DIV\vv{F}=0$, we also have $\vv{F}\in\Hdiv$. Thus $\vv{F}$ has a well-%
defined normal trace and by definition:
\begin{equation}
\intdx{\Gamma_i}{}{\vv{F}\cdot\vv{n}}{S}
= \langle \nu\vv{F},\gamma T_i\rangle_\Gamma 
= \intdx{\Omega}{}{\vv{F}\cdot\nabla T_i}{\vv{x}}+
  \intdx{\Omega}{}{\underbrace{\DIV\vv{F}}_{=0}\,T_i}{\vv{x}}
= \intdx{\Omega}{}{\vv{F}\cdot\nabla T_i}{\vv{x}}.
\end{equation}
\vspace{\baselineskip} 
\end{proof}

\subsection{Uniqueness of Velocity Fields}
\begin{theorem}\label{thm:uniqueness of velocity field}
Let $\Omega\subset\wholespace$ be a bounded, handle-free Lipschitz domain and let
$\vv{F}\in\HOneDualVec$ fulfil the integrability condition~\eqref{eqn:integrability condition}.
Additionally, let $g\in H^{-\frac{1}{2}}(\Gamma)$ be given such that $\langle g,
1\rangle_\Gamma=0$. Then the div-curl system~\eqref{eqn:div-curl-system} has
exactly one solution $\vv{U}\in\LebVec$ with $\vv{U}\cdot\vv{n}=g$ on $\Gamma$.
\end{theorem}
\begin{remark}
The normal trace $\vv{U}\cdot\mathbf{n}$ is well-defined because the solution of
the div-curl system satisfies $\DIV\vv{U}=0$.
\end{remark}
\begin{proof}
	Let us first remark that the condition $\langle g,1\rangle_\Gamma = 0$ is necessary.
	To see this, note that because $\DIV\vv{U}=0$, any solution $\vv{U}\in\vv{L}^2(\Omega)$ of the div-curl system
	\eqref{eqn:div-curl-system} must fulfil:
	\begin{equation}
		\int_\Gamma \vv{U}\cdot\vv{n}\,{\mathrm d}S =
		\int_{\Omega}\DIV\vv{U}\,{\mathrm d}\vv{x} = 0.
	\end{equation}
	
	Now let $\vv{W}\in\LebVec$ denote any solution of the div-curl system,
	whose existence is guaranteed by \Cref{thm:existence of solutions}. Let $P\in
	H^1(\Omega)/\reals$ be the unique solution of the Neumann problem:
	\begin{equation}
		\left\lbrace
		\begin{aligned}
			\Laplace P            &= 0                     &  &\text{in $\Omega$},\\
			\grad    P\cdot\vv{n} &= g - \vv{W}\cdot\vv{n} &  &\text{on $\Gamma$}.
		\end{aligned}
		\right.
	\end{equation}
	Then the function $\vv{U}\coloneqq\vv{W}-\nabla P$ fulfils the conditions of the
	theorem.
	
	To see that it is unique, let $\vv{U}_1,\vv{U}_2\in\LebVec$ denote two solutions of
	the div-curl system~\eqref{eqn:div-curl-system} that fulfil $\vv{U}_1\cdot\vv{n}
	=\vv{U}_2\cdot\vv{n} = g$ on the boundary $\Gamma$. Then their difference
	$\vv{D}\coloneqq\vv{U}_1-\vv{U}_2$ solves
	\begin{equation}
		\left\lbrace
		\begin{aligned}
			\DIV\vv{D} &= 0 &\text{in $\Omega$},\\
			\curl\vv{D} &= \gv{0} &\text{in $\Omega$},\\
			\vv{D}\cdot\vv{n} &= 0\ &\text{on $\Gamma$}.
		\end{aligned}
		\right.
	\end{equation}
	In other words, $\vv{D}$ is a so-called Neumann harmonic field. These functions
	form a space of dimension $\beta_1$.\cite[Proposition~3.14]{amrouche1998}\cite [Section~4.3]{arnold2018} By hypothesis $\beta_1=0$, and
	thus $\vv{D}=\gv{0}$.
\end{proof}

The above proof hints at what needs to be done in order to recover uniqueness in the
case $\beta_1\neq 0$. One needs to prescribe $\beta_1$ functionals
that determine the Neumann harmonic components of $\vv{U}$. A construction of
these fields and corresponding functionals can be found in the work of Amrouche
et al. \cite {amrouche1998}

\section{Stream Functions}
We prove the existence result for stream functions of \Cref{item:streamfct} in this section.
The related uniqueness statement of \Cref{uniqueness item} is proven in \Cref{thm:uniqueness of stream function}.

\subsection{Existence of Stream Functions}
The following theorem is a variant of a result by Girault and Raviart.\cite[%
Theorem~3.4]{raviart1986} We give a different proof, which uses the Newton
operator instead of Fourier transforms. 
\begin{theorem}\label{thm:existence of stream function}
Let $\Omega\subset\wholespace$ denote a bounded Lipschitz domain. Then
$\vv{U}\in\LebVec$ satisfies
\begin{subequations}
\begin{align}
\DIV\vv{U} &= 0,\qquad \text{in $\Omega$},\label{eq:div zero}\\
\int_{\Gamma_i} \vv{U}\cdot\vv{n}\,{\mathrm d}S &= 0,\qquad
i=1,\dotsc,\beta_2,\label{eq:cond int zero on bdy}
\end{align}
\end{subequations}
if and only if there exists a vector-field $\vv{A}\in\HlocVec$ such
that
\begin{equation}\label{eqn:streamfct-conditions}
\left\lbrace
\begin{aligned}
\curl\vv{A}        &=   \vv{U} &  &\text{in $\Omega$},      &\qquad
-\gv{\Delta}\vv{A} &=   \gv{0} &  &\text{in $\wholespace\setminus\overline{\Omega}$}, \\
\DIV\vv{A}         &=       0  &  &\text{in $\wholespace$}, &\qquad
\vv{A}(\vv{x})     &\to \gv{0} &  &\text{as $|\vv{x}|\to\infty$}.
\end{aligned}
\right.
\end{equation}
\end{theorem}
\begin{proof}
\enquote{$\Leftarrow$} Note that the conditions \eqref{eq:div zero} and
\eqref{eq:cond int zero on bdy} are exactly the integrability conditions
\eqref{eqn:intcond1} and \eqref{eqn:intcond3}. Because of \Cref{lmm:plain-existence}
and \Cref{lem:alternative-conditions}, these conditions are necessary to ensure
the existence of a vector-field $\vv{A}\in\LebVec$ such that $\vv{U}=\curl\vv{A}$
in $\Omega$.
	
\enquote{$\Rightarrow$} In order to show sufficiency, the idea is to extend
$\vv{U}$ to  $\wholespace$ by \enquote{potential flows} matching $\vv{U}\cdot%
\vv{n}$ on $\Gamma$, then use the Newton operator.

We want to exploit the following scalar functions. For $i=0$, we let $P_0\in
H^1_{\mathrm{loc}}(\Theta_0)$ denote the solution of the problem:
\begin{equation}
\left\lbrace
\begin{aligned}
 \Laplace P_0  &= 0                              &       &\text{in $\Theta_0$},\\
 \grad    P_0\cdot\vv{n} &= \vv{U}\cdot\vv{n}    &       &\text{on $\Gamma_0$},\\
 P_0(\vv{x})&\to 0                               &       &\text{as $|\vv{x}|\to\infty$,}
\end{aligned}
\right.
\end{equation}
whereas for $i=1,\ldots,\beta_2$ we define $P_i\in H^1(\Omega_i)/\reals$ as
the solution of:
\begin{equation}
\left\lbrace
\begin{aligned}
\Laplace P_i  &= 0                              &       &\text{in $\Theta_i$},\\
\grad    P_i\cdot\vv{n} &= \vv{U}\cdot\vv{n}    &       &\text{on $\Gamma_i$}.\\
\end{aligned}
\right.
\end{equation}
Because of condition~\eqref{eq:cond int zero on bdy}, it is well-known that these
problems are well-posed.

We are now ready to extend $\vv{U}$ to the whole space as
\begin{equation}
\widetilde{\vv{U}}:\wholespace\to\wholespace,\qquad\vv{x}\mapsto
\begin{cases}
\vv{U}(\vv{x})         & \vv{x}\in\Omega, \\
\grad P_i(\vv{x})    & \vv{x}\in\Theta_i,\,i=0,...,\beta_2.
\end{cases}
\end{equation}
Because $\llbracket\widetilde{\vv{U}}\cdot\vv{n}\rrbracket = 0$, we have $\DIV
\widetilde{\vv{U}}=0$ on $\wholespace$. Since $\supp\left(\curl\widetilde{\vv{U}}\right)
\subset\overline{\Omega}$, $\curl\widetilde{\vv{U}}\in\vv{H}^{-1}(\wholespace)$ is
compactly supported, and we may define $\vv{A}\coloneqq\NewtonVec\curl\widetilde{\vv{U}}$.
	
We now claim that $\curl\mathbf{A}=\widetilde{\vv{U}}$ on $\wholespace$.
From the properties of $\NewtonVec$ it follows that $\vv{A}\in\HlocVec$ and
$\vv{A}(\vv{x})\to\gv{0}$ at infinity. Because $\NewtonVec$ commutes with
differentiation, we have $\DIV\vv{A}=\Newton\DIV\curl\widetilde{\vv{U}}=0$
on $\wholespace$, and thus:
\begin{equation}
\left\lbrace
\begin{aligned}
\curl(\curl\vv{A})  &=  & &\Laplacevec\vv{A} &  &= \curl\widetilde{\vv{U}}\\
\DIV(\curl{\vv{A}}) &=  & &\ \ \ \ 0         &  &=\DIV\widetilde{\vv{U}}
\end{aligned}
\right.
\qquad\text{on $\wholespace$}.
\end{equation}
The difference $\vv{D}\coloneqq\curl\mathbf{A}-\widetilde{\vv{U}}$ therefore fulfils:
\begin{equation}
\left\lbrace
\begin{aligned}
\Laplacevec\vv{D}  &=  \gv{0}  &   &\text{on $\wholespace$,}\\
\vv{D}(\vv{x})     &\to\gv{0}  &   &\text{as $|\vv{x}|\to\infty$,}
\end{aligned}
\right.
\end{equation}
so that from \Cref{lem:newton-characterisation} we conclude $\vv{D}=\NewtonVec
\gv{0}=\gv{0}$, that is $\curl\mathbf{A}=\widetilde{\vv{U}}$.
\end{proof}

\subsection{Uniqueness of Stream Functions}
For a given velocity field $\vv{U}$, a stream function $\vv{A}$ can be constructed
as in the proof of \Cref{thm:existence of stream function}. In the following
section, however, we construct a stream function $\vv{A}$ directly from $\vv{F}$,
so there the extended velocity field is \emph{a-priori unknown}. The following
result allows us to establish that the stream functions from \Cref{thm:existence
of stream function} and \Cref{sec:construction} coincide if $\beta_1 = 0$. This
in turn will also allow us to establish higher regularity in \Cref{thm:regularity}.
Note that one obtains uniqueness of $\vv{A}$, while \emph{neither} explicitly
referring to the extension of $\vv{U}$ outside $\Omega$, \emph{nor} to the
boundary values of $\vv{A}$ on $\Gamma$. The proof also gives motivation for the
constructions presented later.

\begin{theorem}\label{thm:uniqueness of stream function}
Let $\Omega\subset\wholespace$ be a handle-free, bounded Lipschitz domain ($\beta_1=0$) and
let $\vv{U}\in\LebVec$ fulfil the conditions of~\Cref{thm:existence of stream function}. Then
there exists exactly one vector-field $\vv{A}\in\HlocVec$ satisfying~\eqref{eqn:streamfct-conditions}.
\end{theorem}
\begin{proof}
Suppose that $\vv{A}_1$ and $\vv{A}_2$ are two vector-fields in $\HlocVec$
satisfying~\eqref{eqn:streamfct-conditions}. Then their difference
$\vv{D}\coloneqq\vv{A}_1-\vv{A}_2\in\HlocVec$ fulfils:
\begin{equation}
\left\lbrace
\begin{aligned}
\curl\vv{D}      &=  \gv{0}  &    &\text{in $\Omega$},      &\quad
\Laplacevec\vv{D}&=  \gv{0}  &    &\text{in $\wholespace\setminus\Gamma$},\\
\DIV\vv{D}       &=      0   &    &\text{in $\wholespace$}, &\quad
\vv{D}(\vv{x})   &\to\gv{0}  &    &\text{as $|\vv{x}|\to\infty$}.
\end{aligned}
\right.
\end{equation}
Thus, \Cref{lem:representation-formula} is applicable, yielding
$\vv{D}=\NewtonVec(\Laplacevec\vv{D})$ and $\Laplacevec\vv{D}=
\gv{\tau}'\llbracket\gv{\rho}\curl{\vv{D}}\rrbracket$. From the mapping
properties of $\gv{\rho}$ it follows that $\vv{s}\coloneqq
\llbracket\gv{\rho}\curl{\vv{D}}\rrbracket\in\HRMinusDiv$. The Hodge
decomposition from \Cref{lem:trace-hodge-decomposition} furthermore yields
\begin{equation}
\vv{s}\coloneqq\llbracket\gv{\rho}\curl{\vv{D}}\rrbracket = \tracecurlvec p - \tracegradp q
\end{equation}
for some functions $p\in H^{\frac{1}{2}}(\Gamma)/\reals$, $q\in H^1(\Gamma)/\reals$
that are uniquely determined up to a constant on each connected part $\Gamma_i$,
$i=0,\dotsc,\beta_2$ of the boundary. It thus suffices to establish that $p=q=0$.
	
We first consider $q$. The fact that $\DIV\vv{D}=0$ on $\wholespace$ implies that
for all $V\in\mathcal{D}(\wholespace)$:
\begin{equation}
\langle\tracediv\vv{s},\gamma V\rangle_\Gamma =
\langle\vv{s},\tracegrad\gamma V\rangle_\Gamma =
\langle\vv{s},\gv{\tau}(\grad V)\rangle_\Gamma =
\langle\Laplacevec\vv{D},\grad V\rangle =
\langle\DIV(\Laplacevec\vv{D}), V\rangle =
\langle\Laplace\DIV\vv{D}, V\rangle = 0,
\end{equation}
that is $\tracediv\vv{s}=0$. This in turn implies $\tracelaplace q =
\tracediv\vv{s} = 0$, and by the coercivity of the Laplace--Beltrami operator
on $H^1(\Gamma)/\reals$ we conclude that $q = 0$.
	
Considering $p$, we note that because $\curl\vv{D}=\gv{0}$ in $\Omega$, we have
$0=(\curl\vv{D})\cdot\vv{n}=\tracecurl\gv{\tau}\vv{D}$ on $\Gamma$. Thus for all
$v\in H^{\frac{1}{2}}(\Gamma)$:
\begin{equation}
0=\langle\tracecurl\gv{\tau}\vv{D},v\rangle_\Gamma
 =\langle\gv{\tau}\vv{D},\tracecurlvec v\rangle_\Gamma
 =\langle\gv{\tau}\NewtonVec\gv{\tau}'\vv{s},\tracecurlvec v\rangle_\Gamma
 =\langle\gv{\tau}\NewtonVec\gv{\tau}'\tracecurlvec p,\tracecurlvec v\rangle_\Gamma.
\end{equation}
The last expression can be enlightened using a more explicit representation.
Following Claeys and Hiptmair,\cite[Equations~(41) and (42)]{claeys2019} under the
additional assumption that $\tracecurlvec p$, $\tracecurlvec v\in\vv{L}^\infty(\Gamma)$,
we have:
\begin{equation}\label{eq:hypersingular BIO}
 \langle\gv{\tau}\NewtonVec\gv{\tau}'\tracecurlvec p,\tracecurlvec v\rangle_\Gamma =
\frac{1}{4\pi}\int_\Gamma\int_\Gamma
\frac{\tracecurlvec p(\vv{y})\cdot\tracecurlvec v(\vv{x})}{|\vv{x}-\vv{y}|}
\,{\mathrm d}S(\vv{y})\,{\mathrm d}S(\vv{x}).
\end{equation}
Here one clearly recognises the hypersingular boundary integral operator
for the scalar Laplace equation.\cite[Section~3.3.4]{sauter2011} This operator
is known to be coercive on $H^{\frac{1}{2}}(\Gamma)/\reals$,\cite[Theorem~3.5.3]{sauter2011}
and we conclude that $p=0$.
\end{proof}

Let us now make some remarks on the case $\beta_1\neq 0$. We define $\Omega^C
\coloneqq\wholespace\setminus\overline{\Omega}$ as the complementary domain of $\Omega$,
and $\vv{B}\coloneqq\curl\vv{D}|_\Omega$ and $\vv{B}^C\coloneqq\curl\vv{D}|_{\Omega^C}$. These
functions are Neumann harmonic fields:
\begin{equation}
\begin{aligned}
	\DIV\vv{B}  &=0, & \curl\vv{B}  &=\gv{0} & &\text{in $\Omega$},   &\quad \vv{B}  \cdot\vv{n}&=0  & &\text{on $\Gamma$},\\
	\DIV\vv{B}^C&=0, & \curl\vv{B}^C&=\gv{0} & &\text{in $\Omega^C$}, &\quad \vv{B}^C\cdot\vv{n}&=0  & &\text{on $\Gamma$}.
\end{aligned}
\end{equation}
Ultimately, the idea is to rely on the fact that in handle-free domains the
space of Neumann harmonic fields only contains the zero element, and thus
$\vv{B}=\gv{0}$ and $\vv{B}^C=\gv{0}$. In the case $\beta_1\neq 0$, however,
neither $\Omega$ nor $\Omega^C$ are handle-free, and in fact we have
$\beta_1^C=\beta_1$. The spaces of Neumann harmonic fields on $\Omega$ and
$\Omega^C$ then each have dimension $\beta_1$.

Buffa has derived the analogue of~\Cref{lem:trace-hodge-decomposition} for the case
of Lipschitz polyhedra with $\beta_1\neq 0$.\cite{buffa2001b} Because $\beta_1
=\beta_1^C$, it contains an additional term from the $2\beta_1$-dimensional
space of harmonic tangential fields. Half of these components are fixed because
of the condition $\curl\vv{A}|_{\Omega}=\vv{U}$, the other half concerns the
external harmonic fields. To ensure uniqueness of~$\vv{A}$, one additionally
needs to prescribe the Neumann harmonic components of $\vv{U}^C\coloneqq%
\curl\vv{A}|_{\Omega^C}$.

\section{Construction of Solutions and Well-posedness}\label{sec:construction}
In this section, we provide a construction for a stream function $\vv{A}%
\in\HlocVec$ for the general case of a given vorticity field
$\vv{F}\in\HOneDualVec$. This construction may also be considered
an alternative proof of the existence results of~%
\Cref{thm:existence of solutions,thm:existence of stream function}.

The idea is to first find a suitable extension $\widetilde{\vv{F}}\in
\vv{H}^{-1}(\wholespace)$ of $\vv{F}$, that additionally satisfies
$\DIV\widetilde{\vv{F}}\in H^{-1}(\wholespace)$. The spurious divergence of
$\widetilde{\vv{F}}$ can then be cancelled out using a surface functional, and
the problem can be solved by applying the Newton operator. In order for this
approach to work it is crucial to make use of \Cref{lmm:identification of spaces}
and to interpret $\vv{F}$ as a member of $\HOcurldual$. Otherwise, one will usually
only obtain $\DIV\widetilde{\vv{F}}\in H^{-2}(\wholespace)$ and the construction
will fail.

In computational practice, one will often have $\vv{F}\in\LebVec$. Under this
assumption we can simplify the construction, yielding an algorithm that is more
easily implementable.

\subsection{General Vorticity Fields}\label{subsec:general-construction}
Let $\vv{F}\in\HOneDualVec$ be given and suppose that it fulfils the integrability
condition \eqref{eqn:integrability condition}, or the equivalent conditions~%
\eqref{eqn:intcond1} and~\eqref{eqn:intcond2}. Furthermore, let $g\in
H^{-\frac{1}{2}}(\Gamma)$ be given boundary data such that $\langle g,
1\rangle_{\Gamma_i}=0$, $i=0,\dots,\beta_2$. Because of \Cref{lmm:identification
of spaces}, we also have $\vv{F}\in\HOcurldual$. Let $\vv{R}\in\HOcurl$ denote
the Riesz representative of $\vv{F}$, i.\,e., the uniquely determined function
$\vv{R}$ such that for all $\vv{V}\in\HOcurl$:
\begin{equation}
\underbrace{\int_\Omega \vv{R}\cdot\vv{V}\,+\,\curl\vv{R}\cdot\curl\vv{V}\,{\mathrm d}\vv{x}}_{%
 =:\mathfrak{B}(\vv{R},\vv{V})}=
\langle\vv{F},\vv{V}\rangle.
\end{equation}

The expression $\mathfrak{B}(\vv{R},\vv{V})$ is not only well-defined
for $\vv{V}\in\HOcurl$, but also for any $\vv{V}\in\vv{H}_0(\curl;\wholespace)$.
We thus define $\widetilde{\vv{F}}$ as follows:
\begin{equation}\label{eqn:tildeF-def}
\forall\vv{V}\in\vv{H}_0(\curl;\wholespace):\ \langle\widetilde{\vv{F}},\vv{V}\rangle
\coloneqq \mathfrak{B}(\vv{R},\vv{V}).
\end{equation}
Obviously $\widetilde{\vv{F}}$ extends $\vv{F}$ and is compactly supported with
$\supp\widetilde{\vv{F}}\subset\overline{\Omega}$. Additionally we also immediately
obtain that $\Vert\widetilde{\vv{F}}\Vert_{\vv{H}^{-1}(\wholespace)}\lesssim
\Vert\vv{F}\Vert_{\vv{H}^{-1}(\Omega)}$. This extension does not necessarily fulfil
$\DIV\widetilde{\vv{F}}=0$ on all of $\wholespace$. However, the following result is
useful.
\begin{lemma}
	One has $\DIV\widetilde{\vv{F}}\in H^{-1}(\wholespace)$. Moreover, there exists
	a uniquely determined surface functional $f\in H^{-\frac{1}{2}}(\Gamma)$
	such that:
	\begin{equation}
		\langle\DIV\widetilde{\vv{F}},V\rangle = -\langle f,\gamma V\rangle_\Gamma
		\qquad\forall V\in H^1(\wholespace),
	\end{equation}
	and
	\begin{equation}
	\left\{
		\begin{aligned}
			&\langle f, 1\rangle_{\Gamma_i} = 0,\qquad i=0,\ldots,\beta_2, \\
			&\Vert f\Vert_{H^{-\frac{1}{2}}(\Gamma)}\lesssim\Vert\vv{F}\Vert_{\HOneDualVec}. &
		\end{aligned}\right.
	\end{equation}
\end{lemma}
\begin{proof}
First note that $\forall V\in\wholetestfcts$:
\begin{equation}
\langle\DIV\widetilde{\vv{F}},V\rangle =
\langle\widetilde{\vv{F}},\grad V\rangle =
\mathfrak{B}(\vv{R},\grad V) =
\int_\Omega\vv{R}\cdot(\grad V)\,{\mathrm d}\vv{x}
\leq\Vert\vv{R}\Vert_{\LebVec}\Vert\grad V\Vert_{\LebVec}
\leq\Vert\vv{F}\Vert_{\HOcurl'}\Vert V\Vert_{H^1(\Omega)}.
\end{equation}
The distribution $\DIV\widetilde{\vv{F}}\in\wholedistris$ thus admits a unique
continuous extension to $\overline{\wholetestfcts}^{\Vert\cdot\Vert_{H^1(%
\wholespace)}}= H^1(\wholespace)$, and we may write $\DIV\widetilde{\vv{F}}%
\in H^{-1}(\wholespace)$ with $\Vert\DIV\widetilde{\vv{F}}\Vert_{H^{-1}(\wholespace)}
\leq \Vert\vv{F}\Vert_{\HOcurl'}\lesssim \Vert\vv{F}\Vert_{\HOneDualVec}$.
	
Next, we find that the value $\langle\DIV\widetilde{\vv{F}},V\rangle$ only
depends on the Dirichlet trace $\gamma V\in H^{\frac{1}{2}}(\Gamma)$ of the
trial function $V\in H^1(\wholespace)$. To see this, let $V_1,V_2\in
H^1(\wholespace)$ have the same Dirichlet trace, $\gamma V_1 = \gamma V_2$.
Because $\gamma(V_1-V_2)=0$, one finds that $-\nabla(V_1-V_2)|_\Omega\in\HOcurl$,
and thus:
\begin{equation}
\langle \DIV\widetilde{\vv{F}}, V_1\rangle -
\langle \DIV\widetilde{\vv{F}}, V_2\rangle =
\mathfrak{B}\bigl(\vv{R},-\nabla( V_1-V_2 ) \bigr) 
= \langle \vv{F}, -\nabla(V_1-V_2)|_\Omega\rangle \stackrel{\eqref{eqn:integrability condition}}{=} 0.
\end{equation}
We may thus define $f\in H^{-\frac{1}{2}}(\Gamma)$ as follows:
\begin{equation}
\forall v\in H^{\frac{1}{2}}(\Gamma):\qquad 
\langle f,v\rangle_\Gamma \coloneqq -\langle\DIV\widetilde{\vv{F}},\gamma^{-1}v\rangle,
\end{equation}
where $\gamma^{-1}: H^{\frac{1}{2}}(\Gamma)\to H^1(\wholespace)$ is fixed, but
may be any linear and bounded lifting operator. Clearly, we have $\Vert f 
\Vert_{H^{-\frac{1}{2}}(\Gamma)}\lesssim\Vert\vv{F}\Vert_{\HOneDualVec}$, because
\begin{equation}
\langle f, v\rangle_\Gamma =  -\langle\DIV\widetilde{\vv{F}},\gamma^{-1}v\rangle
\leq \Vert\vv{F}\Vert_{\HOneDualVec}\Vert\gamma^{-1}v\Vert_{H^1(\wholespace)}
\leq \Vert\vv{F}\Vert_{\HOneDualVec}\Vert\gamma^{-1}\Vert_{H^{\frac{1}{2}}(\Gamma)\to
H^1(\wholespace)}\Vert v\Vert_{H^{\frac{1}{2}}(\Gamma)}
\end{equation}
for all $v\in H^{\frac{1}{2}}(\Gamma)$.
	
Finally, for $i=0,\ldots,\beta_2$, we have
\begin{align}
 \langle f, 1\rangle_{\Gamma_i} = 
 \langle f,\gamma T_i\rangle_\Gamma =
-\langle\DIV\widetilde{\vv{F}},\gamma^{-1}T_i\rangle =
 \langle\vv{F},\nabla\gamma^{-1}T_i|_\Omega\rangle\stackrel{\eqref{eqn:integrability condition}}{=} 0.
\end{align}
\end{proof}

As a consequence of the preceding lemma we may define $q\in H^1(\Gamma)/\reals$,
uniquely up to a constant on each connected component of the boundary
$\Gamma_i$, $i=0,\ldots,\beta_2$, as the solution to the Laplace--Beltrami
equation:
\begin{equation}
	\tracelaplace q = f\qquad\text{on $\Gamma$},
\end{equation}
and furthermore define
$\widehat{\vv{F}}\coloneqq\widetilde{\vv{F}} - \gv{\tau}'\nabla_\Gamma q$,
and
$\widehat{\vv{A}}\coloneqq\NewtonVec\widehat{\vv{F}}.$
\begin{lemma}\label{lem:fhat}
One has $\widehat{\vv{F}}\in\vv{H}^{-1}(\wholespace)\cap\compdistrisvec$,
$\widehat{\vv{A}}\in\HlocVec$ decaying to zero at infinity, and moreover:
\begin{equation}
\left\lbrace
\begin{aligned}
 \Vert\widehat{\vv{A}}\Vert_{\HOneVec}                 \lesssim\;
&\Vert\widehat{\vv{F}}\Vert_{\vv{H}^{-1}(\wholespace)} \lesssim
 \Vert\vv{F}\Vert_{\HOneDualVec}, &\qquad
\Laplacevec\widehat{\vv{A}} &= \widehat{\vv{F}} = \vv{F} &   &\text{in }\Omega,\\
\DIV\widehat{\vv{A}} = &\DIV\widehat{\vv{F}}=0\quad \text{in }\wholespace, &
\Laplacevec\widehat{\vv{A}} &= \widehat{\vv{F}} = \vv{0} & &\text{in }\wholespace\setminus\overline{\Omega}.
\end{aligned}
\right.
\end{equation}
\end{lemma}
\begin{proof}
It suffices to establish the properties of $\widehat{\vv{F}}$; the results for
$\widehat{\vv{A}}$ then immediately follow from the properties of $\NewtonVec$.
We already established that $\widetilde{\vv{F}}\in\vv{H}^{-1}(\wholespace)$ and
$\Vert\widetilde{\vv{F}}\Vert_{\vv{H}^{-1}(\wholespace)} \lesssim
\Vert\vv{F}\Vert_{\HOneDualVec}$. For the surface functional,
\begin{equation}
\Vert\nabla_\Gamma q\Vert_{\vv{L}^2(\Gamma)}\leq
\Vert q\Vert_{H^1(\Gamma)}\lesssim
\Vert f\Vert_{H^{-1}(\Gamma)}\lesssim
\Vert f\Vert_{H^{-\frac{1}{2}}(\Gamma)}\lesssim
\Vert\vv{F}\Vert_{\HOneDualVec},
\end{equation}
so that $\forall\vv{V}\in\wholetestfctsvec$:\ \ 
\begin{math}
\langle \nabla_\Gamma q, \gv{\tau}\vv{V}\rangle_\Gamma \lesssim
\Vert\vv{F}\Vert_{\HOneDualVec}\Vert\vv{V}\Vert_{\vv{H}^1(\wholespace)}.
\end{math}
Thus $-\gv{\tau}'\nabla_\Gamma q\in\vv{H}^{-1}(\wholespace)$ with 
\begin{math}
\Vert-\gv{\tau}'\nabla_\Gamma q\Vert_{\vv{H}^{-1}(\wholespace)}
\lesssim\Vert\vv{F}\Vert_{\HOneDualVec}.
\end{math}
	
The fact that $\widehat{\vv{F}}=\gv{0}$ on $\wholespace\setminus\overline{\Omega}$
and $\widehat{\vv{F}}=\vv{F}$ on $\Omega$ is obvious. For the divergence, we note
that $\forall V\in\wholetestfcts$:
\begin{align}
\langle \nabla_\Gamma q, \gv{\tau}\nabla V\rangle_\Gamma =
\langle \nabla_\Gamma q, \nabla_\Gamma\gamma V\rangle_\Gamma =
\langle -\Delta_\Gamma q,\gamma V\rangle_\Gamma =
\langle f,\gamma V\rangle_\Gamma =
 -\langle\DIV\widetilde{\vv{F}},V\rangle, 
\end{align}
and therefore $\DIV\widehat{\vv{F}}=0$.
\end{proof}

With these properties in place, it immediately follows that
$\widehat{\vv{U}}\coloneqq\curl\widehat{\vv{A}}$ solves the div-curl system~%
\eqref{eqn:div-curl-system}, but does not necessarily fulfil
$\widehat{\vv{U}}\cdot\vv{n}=g$ on $\Gamma$.  To fix its normal
component, it then suffices to solve the hypersingular boundary
integral equation:
\begin{equation}\label{eqn:hypersingular-bie}
	\forall v\in H^{\frac{1}{2}}(\Gamma):\quad
	\langle\gv{\tau}\NewtonVec\gv{\tau}'\tracecurlvec p,\tracecurlvec v\rangle_\Gamma
	= \langle g-\curl\widehat{\vv{A}}\cdot\mathbf{n},v\rangle_\Gamma,
\end{equation}
for the unknown $p\in H^{\frac{1}{2}}(\Gamma)/\reals$. This problem is known
to be well-posed, and its solution continuously depends on $\widehat{\vv{U}}\cdot\vv{n}$
and $g$:\cite[Theorem~3.5.3]{sauter2011}
\begin{equation}
	\Vert p\Vert_{H^{\frac{1}{2}}(\Gamma)} \lesssim
	\Vert\widehat{\vv{U}}\cdot\vv{n} - g\Vert_{H^{-\frac{1}{2}}(\Gamma)}\lesssim
	\Vert\widehat{\vv{U}}\Vert_{\Hdiv} + \Vert g\Vert_{H^{-\frac{1}{2}}(\Gamma)}
	\lesssim\Vert\widehat{\vv{A}}\Vert_{\vv{H}^1(\Omega)} + \Vert g\Vert_{H^{-\frac{1}{2}}(\Gamma)}
	\lesssim\Vert\vv{F}\Vert_{\HOneDualVec} + \Vert g\Vert_{H^{-\frac{1}{2}}(\Gamma)}.
\end{equation}

We now finally define:
\begin{equation}\label{eqn:construction}
\left\lbrace
\begin{aligned}
\vv{s} &\coloneqq \tracecurlvec p-\nabla_\Gamma q\in\HRMinusDiv,\\
\vv{A} &\coloneqq\NewtonVec\bigl(\widetilde{\vv{F}}+\gv{\tau}'\vv{s}\bigr)\in\HlocVec,\\
\vv{U} &\coloneqq\curl\vv{A}\in\vv{L}^2(\wholespace).
\end{aligned}
\right.
\end{equation}

Then $\vv{U}$ solves the div-curl system and $\vv{U}\cdot\vv{n}=g$ on $\Gamma$,
and $\vv{A}$ is a stream function for $\vv{U}$. In the case of a handle-free
domain, \Cref{thm:uniqueness of velocity field,thm:uniqueness of stream function}
guarantee that these functions are unique. Moreover, the solution continuously
depends on $\vv{F}$ and $g$. In total we have therefore proven the following theorem.

\begin{theorem}\label{thm:well-posedness}
Let $\vv{F}\in\HOneDualVec$ be given and fulfil the integrability condition
\eqref{eqn:integrability condition}, or the equivalent conditions~%
\eqref{eqn:intcond1} and~\eqref{eqn:intcond2}. Furthermore, let $g\in
H^{-\frac{1}{2}}(\Gamma)$ be given, such that $\langle g, 1\rangle_{\Gamma_i}=0$
for all $i=0,\dots,\beta_2$.
	
Then a solution to the div-curl system~\eqref{eqn:div-curl-system} with
$\vv{U}\cdot\vv{n}=g$ on $\Gamma$, and its associated stream function $\vv{A}$ are
given by \eqref{eqn:construction}. In case of a handle-free domain $\vv{U}$
and $\vv{A}$ are the uniquely determined functions from
\Cref{thm:uniqueness of velocity field,thm:uniqueness of stream function}.

These functions linearly and continuously depend on the data $\vv{F}$ and $g$,
and we have:
\begin{equation}
\Vert\vv{U}\Vert_{\LebVec}\lesssim\Vert\vv{A}\Vert_{\vv{H}^1(\Omega)}
\lesssim \Vert\vv{F}\Vert_{\HOneDualVec} + \Vert g\Vert_{H^{-\frac{1}{2}}(\Gamma)}.
\end{equation}
\end{theorem}


\subsection{Square-integrable Vorticity Fields}
In case we actually have $\vv{F}\in\LebVec$, the construction can be simplified.
Thus, let $\vv{F}\in\LebVec$ fulfil the integrability conditions~%
\eqref{eqn:intcond1} and~\eqref{eqn:intcond3}. We first note that $\vv{F}$ now
possesses a natural extension by zero:
\begin{equation}\label{eqn:zero-extension-l2}
	\widetilde{\vv{F}}: \wholespace\to\wholespace,\quad\vv{x}\mapsto
	\begin{cases}
		\vv{F}(\vv{x}) & \vv{x}\in\Omega,\\
		\gv{0}         & \text{else.}
	\end{cases}
\end{equation}

Next, we note that because $\vv{F}\in\LebVec$ and $\DIV\vv{F}=0$ in $\Omega$, we
also have $\vv{F}\in\Hdiv$. Thus $\vv{F}$ has a normal trace $\vv{F}\cdot\vv{n}\in
H^{-\frac{1}{2}}(\Gamma)$, that by condition~\eqref{eqn:intcond3} satisfies
$\langle \vv{F}\cdot\vv{n},1\rangle_{\Gamma_i}=0$, $i=0,\ldots,\beta_2$. One
then has $\DIV\widetilde{\vv{F}} = -\gamma'(\vv{F}\cdot\vv{n})$, so we may instead
define $q\in H^1(\Gamma)/\reals$ as the solution to the Laplace--Beltrami equation
\begin{equation}
\tracelaplace q = \vv{F}\cdot\vv{n}\qquad\text{on $\Gamma$}.
\end{equation}

We then let $\widehat{\vv{F}}\coloneqq\widetilde{\vv{F}}-\gv{\tau}'\nabla_\Gamma q$
as before and note that \Cref{lem:fhat} holds. The term $\gv{\tau}'(\tracegrad q)$
is the correction to the Biot--Savart law mentioned \cref{sec:problematic-approaches}.
From this point the construction proceeds as before.

\section{Regularity}
We begin this section by recalling a result of Costabel.\cite{costabel1990}
\begin{lemma}\label{lem:costabel-regularity}
Let $\Omega\subset\wholespace$ denote a handle-free, bounded Lipschitz
domain and let $\vv{U}\in\LebVec$ fulfil:
\begin{equation}
\DIV\vv{U}\in L^2(\Omega),\ \curl\vv{U}\in\LebVec.
\end{equation}
Then $\vv{U}$ satisfies $\vv{U}\cdot\vv{n}\in L^2(\Gamma)$ on $\Gamma$ if and
only if $\vv{U}\times\vv{n}\in\vv{L}^2(\Gamma)$; and in this case $\vv{U}$
fulfils $\vv{U}\in\vv{H}^{\frac{1}{2}}(\Omega)$.
\end{lemma}
\begin{remark}
There are extensions of this result to domains with $\beta_1\neq0$, for
example in Monk's book.\cite[Theorem~3.47]{monk2003} However, this extension 
is lacking the statement $\vv{U}\cdot\vv{n}\in L^2(\Gamma)\iff
\vv{U}\times\vv{n}\in\vv{L}^2(\Gamma)$ which we need for our proof below. We
thus refer to the original work of Costabel for $\beta_1 = 0$, but one can
expect that this equivalence also generalises to the case $\beta_1\neq 0$.
Under this assumption the following regularity result remains true if 
zero Neumann harmonic components are prescribed for $\vv{U}^C =
\curl\vv{A}|_{\wholespace\setminus\overline{\Omega}}$.
\end{remark}

This result can directly be applied to velocity fields $\vv{U}$ solving
the div-curl system~\eqref{eqn:div-curl-system}. In the following, we show
that it also implies higher regularity of the associated stream functions
$\vv{A}$.

\begin{theorem}\label{thm:regularity}
Let $\Omega\subset\wholespace$ be a bounded, handle-free Lipschitz domain.
Let $\vv{F}\in\LebVec$ be given and fulfil the integrability condition~%
\eqref{eqn:integrability condition}, or the equivalent conditions~%
\eqref{eqn:intcond1} and~\eqref{eqn:intcond3}. Furthermore, let $g\in
L^2(\Gamma)$ be given, such that $\langle g, 1\rangle_{\Gamma_i}=0$
for all $i=0,\dots,\beta_2$.
	
Then, the unique solution $\vv{U}\in\LebVec$ of the div-curl system~\eqref{eqn:div-curl-system}
with $\vv{U}\cdot\vv{n}=g$ on~$\Gamma$ fulfils $\vv{U}\in\vv{H}^{\frac{1}{2}}(\Omega)$,
and its uniquely determined stream function $\vv{A}$ from~\Cref{thm:uniqueness of stream function}
fulfils $\vv{A}\in\vv{H}^{\frac{3}{2}}_{\mathrm{loc}}(\wholespace\setminus\Gamma)$.
\end{theorem}
\begin{proof}
The regularity of $\vv{U}$ is exactly Costabel's result~\Cref{lem:costabel-regularity}.	
For the regularity of $\vv{A}$, we first consider the case $g=0$. Thus, let
$\vv{U}_0$ denote the unique solution of the div-curl system~%
\eqref{eqn:div-curl-system} that satisfies $\vv{U}_0\cdot\vv{n}=0$ on $\Gamma$,
and let $\vv{A}_0\in\HlocVec$ denote its associated stream
function. An application of the representation formula for the vector Laplacian
then yields:
\begin{equation}
\left\lbrace
\begin{aligned}
\vv{A}_0 &= \NewtonVec(\widetilde{\vv{F}}+\gv{\tau}'\vv{s}_0) & &\text{on $\wholespace$}, \\
\vv{s}_0 &= \llbracket\curl\vv{A}_0\times\vv{n}\rrbracket     & &\text{on $\Gamma$},
\end{aligned}\right.
\end{equation}
where $\widetilde{\vv{F}}\in\vv{L}^2(\wholespace)$ is $\vv{F}$'s zero extension
as defined in \eqref{eqn:zero-extension-l2}. Clearly, from the mapping properties
of the Newton operator, it follows that
$\NewtonVec\widetilde{\vv{F}}\in\vv{H}^2_{\mathrm{loc}}(\wholespace)$.
For the boundary term we note that from the construction of $\vv{A}_0$ in the
proof of \Cref{thm:existence of stream function} it is clear that $\curl\vv{A}_0
=\gv{0}$ in $\wholespace\setminus\overline\Omega$. This implies that
\begin{equation}
\vv{s}_0 = \vv{U}_0\times\vv{n}\qquad\text{on $\Gamma$},
\end{equation}
and because of \Cref{lem:costabel-regularity} this yields $\vv{s}_0\in\vv{L}^2
(\Gamma)$. The boundary term $\NewtonVec\gv{\tau}'\vv{s}_0$ may thus
alternatively be interpreted as a component-wise application of the scalar
single layer potential operator $\Newton\gamma'$ to the components of $\vv{s}_0$.
For this operator the following mapping property is known:%
\cite[Remark~3.1.18b]{sauter2011}
\begin{equation}
\Newton\gamma': L^2(\Gamma)\to H^{\frac{3}{2}}_{\mathrm{loc}}(\wholespace\setminus\Gamma),
\end{equation}
and thus $\NewtonVec\gv{\tau}'\vv{s}_0\in\vv{H}^{\frac{3}{2}}_{\mathrm{loc}}(\wholespace\setminus\Gamma)$.
	
For general boundary data $\vv{U}\cdot\vv{n}=g\in L^2(\Gamma)$, one then needs
to solve the hypersingular boundary integral equation:
\begin{equation}
\forall v\in H^{\frac{1}{2}}(\Gamma):\qquad
\langle\gv{\tau}\NewtonVec\gv{\tau}'\tracecurlvec p,\tracecurlvec v\rangle_\Gamma
= \langle g,v\rangle_\Gamma,
\end{equation}
for the unknown $p\in H^{\frac{1}{2}}(\Gamma)/\reals$ and set $\vv{s}
\coloneqq \vv{s}_0 + \tracecurl p$, $\vv{A}\coloneqq\NewtonVec(\widetilde{\vv{F}}+
\gv{\tau}'\vv{s})$. For the integral equation the following regularity
result is known:\cite[Theorem~3.2.3b]{sauter2011}
\begin{equation}
g\in L^2(\Gamma)\Longrightarrow p\in H^1(\Gamma).
\end{equation}
Thus $\tracecurl p\in\vv{L}^2(\Gamma)$ and by the same arguments as above one
obtains that $\vv{A}\in\vv{H}^{\frac{3}{2}}_{\mathrm{loc}}(\wholespace\setminus\Gamma)$.
\end{proof}

\section{Numerical Approximations}\label{sec:numerics}
The case where $\vv{F}\in\vv{L}^r(\Omega)$ for some $r>2$, leads to a
particularly simple discretisation using Raviart--Thomas elements. In practice this
condition on $\vv{F}$ is often not a real restriction. Here, we will only give a
brief sketch of this scheme and its analysis. The case $\vv{F}\in\HOneDualVec$ is
technically more involved, and we restrict ourselves to giving some remarks on
possible numerical realisations. At the end of this section we give a numerical
example with $\vv{F}\in\vv{C}^\infty(\Omega)$ and $\vv{U}\in\vv{C}^\infty(\Omega)$.
Even for such smooth data, the tangential $\vv{A}_{\mathrm{T}}\in\Hcurl\cap\HOdiv$
shows quite strong singularities. The newly proposed stream function
$\vv{A}\in\HOneVec$ is not smooth either, but displays increased regularity
compared to $\vv{A}_{\mathrm{T}}$.
 
\subsection{Meshes and Spaces}
For our numerical approximations we will assume that the domain $\Omega$ is polyhedral,
handle-free, and that a family $\lbrace\mathcal{T}_h\rbrace_{h>0}$ of shape-regular,
quasi-uniform, tetrahedral meshes is available. We will write $T\in\mathcal{T}_h$ for
the tetrahedra of such a mesh, the parameter $h>0$ refers to the average diameter of
these tetrahedra. On these meshes we respectively define standard Lagrangian elements,
N\'ed\'elec elements, Raviart--Thomas elements, and discontinuous elements of
\emph{order} $n\in\mathbb{N}$:
\begin{equation}
\begin{aligned}
\FEMcont &\coloneqq \lbrace V_h\in\HOne\,|\,\forall T\in\mathcal{T}_h:
V_h|_T\in\mathbb{P}_{n-1} \rbrace, \\
\FEMNED &\coloneqq \lbrace\vv{V}_h\in\Hcurl\,|\,\forall T\in\mathcal{T}_h:
\vv{V}_h|_T = \vv{a}+\vv{b},\ \vv{a}\in\mathbb{P}_{n-1}^3,\ \vv{b}\in\overline{\mathbb{P}}_{n}^3,\
\vv{b}(\vv{x})\cdot\vv{x}\equiv 0\rbrace, \\ 
\FEMRT &\coloneqq \lbrace\vv{V}_h\in\Hdiv\,|\,\forall T\in\mathcal{T}_h:
\vv{V}_h|_T = \vv{a} + \vv{x}b, \vv{a}\in\mathbb{P}_{n-1}^3, b\in\overline{\mathbb{P}}_{n-1}\rbrace, \\
\FEMdisc &\coloneqq \lbrace V_h\in L^2(\Omega)\,|\,\forall T\in\mathcal{T}_h:
V_h|_T\in\mathbb{P}_{n-1} \rbrace.
\end{aligned}
\end{equation}
Here, $\mathbb{P}_{n-1}$ refers to the space of polynomials of total \emph{degree}
$n-1$ or less, and $\overline{\mathbb{P}}_{n-1}$ to the space homogeneous
polynomials of  total degree exactly $n-1$. We also make use of sub-spaces with
zero boundary conditions:
\begin{equation}
\begin{alignedat}{10}
\FEMOcont &\coloneqq &\  &\lbrace &    V_h  &\in\FEMcont & \,&|\, & \gamma      V_h  &=0\rbrace, \\
\FEMNEDO  &\coloneqq &\  &\lbrace & \vv{V}_h&\in\FEMNED  & \,&|\, & \gv{\tau}\vv{V}_h&=\boldsymbol{0}\rbrace, \\ 
\FEMRTO   &\coloneqq &\  &\lbrace & \vv{V}_h&\in\FEMRT   & \,&|\, & \nu      \vv{V}_h&=0\rbrace.
\end{alignedat}
\end{equation}

A family of tetrahedral meshes $\lbrace\mathcal{T}_h\rbrace_{h>0}$ automatically
gives rise to a family of boundary triangulations $\lbrace\pd\mathcal{T}_h\rbrace_{h>0}$,
consisting of triangles $t\in\pd\mathcal{T}_h$. On these boundary meshes we will
make use of the boundary element spaces $\BEMcont$, $\BEMRT$, and $\BEMdisc$,
which are the natural analogues of their respective counterparts on the domain
$\Omega$.

\subsection{The Case $\vv{F}\in\vv{L}^r(\Omega)$ with $r>2$}\label{sec:rt-method}
Let $\vv{F}\in\vv{L}^2(\Omega)$ be given, and let $\vv{F}$ fulfil the integrability
conditions $\DIV\vv{F}=0$ and $\langle\vv{F}\cdot\vv{n},1\rangle_{\Gamma_i}=0$,
$i=1,\dotsc,\beta_2$. Let us furthermore assume that there exists an $r>2$ such
that $\vv{F}\in\vv{L}^r(\Omega)$. This condition in particular allows us to define
$\vv{F}_h\in\FEMRT$ as the \emph{canonical interpolant} of $\vv{F}$.\cite[Section~III.3]{brezzi1991}
Note that then $\vv{F}_h$ \emph{also fulfils the integrability conditions}, and
furthermore the standard interpolation error-bound:
\begin{equation}
\Vert\vv{F}-\vv{F}_h\Vert_{\LebVec} = \Vert\vv{F}-\vv{F}_h\Vert_{\Hdiv}
\lesssim h^s\vert\vv{F}\vert_{\vv{H}^s(\Omega)} \qquad 1\leq s\leq n.
\end{equation}
Using the boundedness of the canonical interpolator on $\vv{L}^r(\Omega)\cap\Hdiv$,
we furthermore easily obtain by standard arguments that for $n\geq 2$:
\begin{equation}\label{eqn:finterpolation-error}
\Vert\vv{F}-\vv{F}_h\Vert_{\HOneDualVec} \lesssim h^{s+1}\times
\begin{cases}
\Vert\vv{F}\Vert_{\vv{W}^{s,r}(\Omega)} & \text{if } 0\leq s\leq 1, \\
\vert\vv{F}\vert_{\vv{H}^{s}(\Omega)}   & \text{if } 1\leq s\leq n.
\end{cases}
\end{equation}
Let us denote by $\widetilde{\vv{F}}_h$ the zero extension of $\vv{F}_h$ to
$\wholespace$. The Newton potential $\NewtonVec\widetilde{\vv{F}}_h$ can then be
evaluated analytically and component-wise, as $\vv{F}_h$ is a piece-wise
polynomial on a simplical mesh.

Next we seek the correction $\vv{r}\coloneqq\tracegrad q$ on the boundary, to
cancel out $\DIV\widetilde{\vv{F}}_h$. It is possible to solve the Laplace--%
Beltrami equation $\tracelaplace q_h = \vv{F}_h\cdot\vv{n}$ using a standard Galerkin
method on the space $\BEMcont/\reals$. The convergence rate of such an approach
would then depend on the regularity of $q$, which in turn non-trivially depends
on the shape of the boundary $\Gamma$.\cite[Theorem~8]{buffa2002b} Note, however,
that we have $\vv{F}_h\cdot\vv{n}\in\BEMdisc/\reals$, because $\vv{n}$ is constant
on each triangle $t\in\pd\mathcal{T}_h$. Furthermore we have $\tracediv\BEMRT=\BEMdisc/
\reals$, so it makes sense to use the mixed formulation with Raviart--Thomas
boundary elements instead. Hence, we seek $(\vv{r}_h,q_h)\in\BEMRT\times\BEMdisc/\reals$
such that:
\begin{equation}
\begin{aligned}
\forall \vv{v}_h\in\BEMRT&:      & \int_{\Gamma} \vv{r}_h\cdot\vv{v}_h{\mathrm d}S
 - \int_{\Gamma} q_h(\tracediv\vv{v}_h){\mathrm d}S &= 0, \\
\forall l_h\in\BEMdisc/\reals&:  & \int_{\Gamma} (\tracediv\vv{r}_h)l_h\,{\mathrm d}S
 &= \int_\Gamma (\vv{F}_h\cdot\vv{n}) l_h\,{\mathrm d}S.
\end{aligned}
\end{equation}
It is well-known that this formulation is well-posed, and that moreover we have
$\tracediv\vv{r}_h = \vv{F}_h\cdot\vv{n}$ exactly. Thus, abbreviating
$\widehat{\vv{F}}_h\coloneqq\widetilde{\vv{F}}_h+\gv{\tau}'\vv{r}_h$, we have
$\DIV\widehat{\vv{F}}_h = 0$ on $\wholespace$ \emph{exactly}. The Newton potential
of $\gv{\tau}'\vv{r}_h$ is the component-wise application of the standard,
scalar single layer potential operator $\Newton\gamma'$ to the components of
$\vv{r}_h$ and can also be computed efficiently analytically.

It remains to compute $p_h$. For this we abbreviate $\widehat{\vv{A}}_h :=
\NewtonVec(\widetilde{\vv{F}}_h+\gv{\tau}'\vv{r}_h)$ and seek $p_h\in\BEMcont/
\reals$ such that:
\begin{equation}
\forall v_h\in\BEMcont/\reals:
\frac{1}{4\pi}\int_\Gamma\int_\Gamma
\frac{\tracecurlvec p_h(\vv{y})\cdot\tracecurlvec v_h(\vv{x})}{|\vv{x}-\vv{y}|}
\,{\mathrm d}S(\vv{y})\,{\mathrm d}S(\vv{x}) =
\int_{\Gamma}\left(g-\curl\widehat{\vv{A}}_h\cdot\vv{n}\right)v_h\,{\mathrm d}S.
\end{equation}
This is a standard Galerkin boundary element method, analysis and efficient
implementation techniques can for example be found in the book of Sauter and Schwab.\cite{sauter2011}
The numerical approximation is then defined as $\vv{s}_h\coloneqq\tracecurlvec
p_h+\vv{r}_h$, and $\vv{A}_h\coloneqq\NewtonVec\widetilde{\vv{F}}_h+\gv{\tau}'\vv{s}_h$.

Let us define $g_h\coloneqq\vv{n}\cdot\curl\vv{A}_h$. Then $\vv{A}_h$ is the
\emph{exact solution to the perturbed problem} with data $\vv{F}_h$ and $g_h$,
so by the well-posedness result of \Cref{thm:well-posedness}, we immediately
obtain:
\begin{equation}\label{eqn:aposteriori-error}
\Vert\curl\vv{A}_h-\vv{U}\Vert_{\LebVec} \lesssim
\Vert\vv{A}_h-\vv{A}\Vert_{\HOneVec} \lesssim
\Vert\vv{F}_h-\vv{F}\Vert_{\HOneDualVec} + \Vert g_h-g\Vert_{H^{-\frac{1}{2}}(\Gamma)}.
\end{equation}

 For the first part $\Vert\vv{F}-\vv{F}_h\Vert_{\HOneDualVec}$ we can use
\eqref{eqn:finterpolation-error}. Note that this part of the error neither depends
on the regularity of $\vv{A}$ nor on that of $\vv{U}$! For smooth data $\vv{F}$
one may thus use coarse meshes and high order $n$ on $\Omega$. All of the
\enquote{irregularity} due to the non-smooth boundaries is \enquote{concentrated}
in the term $\Vert g-g_h\Vert$, and for $g\in L^2(\Gamma)$, as a consequence of
the regularity result~\Cref{thm:regularity}, we may at most expect convergence of
order $\mathcal{O}(h^\frac{1}{2})$. However, also note that
\eqref{eqn:aposteriori-error} is an \emph{a-posteriori} bound that is easily
computable. The term $\Vert g - g_h\Vert$, for example, can be made arbitrarily
small by increasing the mesh resolution for $p_h$. Because $p_h$ is  the solution
to a \emph{scalar} equation on the boundary, adaptive refinement strategies are
a comparatively cheap remedy to tackle the irregularity.

\subsection{Remarks on the General Case $\vv{F}\in\HOneDualVec$}
It is straightforward to discretise the general construction given in
\Cref{subsec:general-construction}. An obvious choice would be a standard
Galerkin method and finding $\vv{R}_h\in\FEMNEDO$ such that $\mathfrak{B}%
(\vv{R}_h,\vv{V}_h)=\langle\vv{F},\vv{V}_h\rangle$ for all $\vv{V}_h\in\FEMNEDO$.
Ultimately, however, the rate of convergence for such a method would less depend
on the regularity of $\vv{F}$ itself, but more on that of its Riesz representative
$\vv{R}$. If $\vv{F}$ happens to be smooth, $\vv{R}$ can turn out to be much less
regular than the function $\vv{F}$ it represents.

An alternative approach could be based on the normal potential $\vv{A}_{\mathrm{N}}
\in\HOcurl\cap\Hdiv$ for the velocity field $\vv{U}_0\in\LebVec$ with $\vv{U}_0
\cdot\vv{n}=0$ on $\Gamma$. In their work Amrouche et al. describe a finite element
method for its approximation. For example, for a hole-free domain ($\beta_2=0$),
their method reads as follows. Find
$(\vv{A}_{\mathrm{N},h},P_h)\in\FEMNEDO\times S_{h,0}^n(\Omega)$ such that:%
\cite[Equation~(4.13)]{amrouche1998}
\begin{equation}
\begin{aligned}
\forall\vv{V}_h\in\FEMNEDO&: &
\int_{\Omega}\curl\vv{A}_{\mathrm{N},h}\cdot\curl\vv{V}_{h}\,{\mathrm d}\vv{x} +
\int_{\Omega}\vv{V}_{h}\cdot\gradp P_h\,{\mathrm d}\vv{x} &=
\int_{\Omega}\vv{U}_0\cdot\curl\vv{V}_h\,{\mathrm d}\vv{x}, \\
\forall Q_h\in S_{h,0}^n(\Omega)&: &
\int_{\Omega}\vv{A}_{\mathrm{N},h}\cdot\gradp Q_h\,{\mathrm d}\vv{x} &= 0.
\end{aligned}
\end{equation}
This scheme has the remarkable property that $\Vert\curl\vv{A}_{\mathrm{N},h}-
\vv{U}_0\Vert_{\LebVec}\lesssim h^s\Vert\vv{U}_0\Vert_{\vv{H}^{s}(\Omega)}$ for
all $0\leq s\leq n$, regardless of the regularity of $\vv{A}_{\mathrm{N}}$.
Only minor modifications are necessary for the case $\beta_2>0$.

Note that the right side may be immediately replaced with $\langle\vv{F},\vv{V}_h
\rangle$, as $\vv{V}_h\in\HOcurl$ and $\vv{F}\in\HOcurldual$ by \Cref{lmm:identification
of spaces}. This is not possible in the corresponding method for the tangential
potential $\vv{A}_{\mathrm{T}}$, where the test functions do not have a vanishing
tangential trace.

The true solution fulfils $\DIV\vv{A}_{\mathrm{N}} = 0$ in $\Omega$, it then
suffices to cancel the spurious divergence of its zero extension $\widetilde{%
\vv{A}}_{\mathrm{N}}$. We have $\DIV\widetilde{\vv{A}}_{\mathrm{N}} =
-\gamma'\vv{A}_{\mathrm{N}}\cdot\vv{n}$, so we may solve $\tracelaplace q =
\vv{A}_{\mathrm{N}}\cdot\vv{n}$, set $\vv{r}\coloneqq\tracegrad q$, and obtain
$\widehat{\vv{A}}\coloneqq\widetilde{\vv{A}}_{\mathrm{N}}+\NewtonVec\gv{\tau}'
\vv{r}\in\vv{H}^1_{\mathrm{loc}}(\wholespace)$. The numerical approximation
$\vv{A}_{\mathrm{N},h}$ will usually not be exactly divergence free, here the
jumps of the normal traces on the internal faces will also need to be cancelled
out in order to achieve $\HOneVec$-regularity. From here  we may proceed
analogously as before: for the correction of the normal trace of $\curl
\widehat{\vv{A}}$ we solve a boundary integral equation.

\subsection{An Example Illustrating Higher Regularity}\label{sec:numerical-example}
We consider the domain $\Omega\coloneqq(0,1)^3\setminus[0.1,0.8]^3$ and the
smooth velocity field $\vv{U}\in\vv{C}^\infty(\Omega)$ associated to the
vorticity $\vv{F}\in\vv{C}^\infty(\Omega)$ given by:
\begin{equation}
\vv{U}(\vv{x}) \coloneqq \frac{1}{2}\begin{pmatrix}-x_2\\x_1\\0\end{pmatrix},\qquad
\vv{F}(\vv{x}) \coloneqq \curl\vv{U}(\vv{x}) = \begin{pmatrix}0\\0\\1\end{pmatrix}.
\end{equation}
The domain $\Omega$ was chosen to be asymmetric, non-smooth, non-convex, and
topologically non-trivial ($\beta_2=1$), while at the same time being
\enquote{easy} from the viewpoint of meshing.

Neither for the tangential potential $\vv{A}_\mathrm{T}$, nor for the potential
introduced in this work explicit expressions are known. Thus, the finite element
method by Amrouche et al. has been implemented to compute $\vv{A}_\mathrm{T}$.%
\cite[Equation (4.5)]{amrouche1998} We use order $n=2$ for
$\vv{A}_{\mathrm{T},h}\in\FEMNED$, such that the velocity field is recovered
exactly.

For the other stream function, notice that in this case the Newton potential:
\begin{equation}
\NewtonVec\widetilde{\vv{F}}(\vv{x}) =
\frac{1}{4\pi}\int_{\Omega}\frac{{\mathrm d}\vv{y}}{|\vv{x}-\vv{y}|}\,\cdot\,
\begin{pmatrix}0\\0\\1\end{pmatrix}
\end{equation}
can be evaluated directly, without further discretisation.\cite{kirchhart2020b}
The Laplace--Beltrami equation for $q_h$ and $\vv{r}_h$, as well as the
hypersingular boundary integral equation for $p_h$, are discretised as described
above, with order $n=2$.

We remark that the method by  Amrouche et al. requires to use the velocity field
$\vv{U}$ as input, while the approach discussed in this work only requires $\vv{F}$
and the boundary data $\vv{U}\cdot\vv{n}$.

The numerical results along the line $(x_1,0.5,0.8)^\top$ are shown in
\Cref{fig:lineplots}{\kern-5pt}. In the interval $x_1\in[0.1,0.8]$ this
line touches the internal boundary $\Gamma_1$. One clearly sees that close
to the corners at $x_1=0.1$ and $x_1=0.8$ neither solutions are smooth. But
while the tangential potential $\vv{A}_{\mathrm{T}}$ develops very steep
gradients and jumps near $x_1=0.1$, the new potential $\vv{A}$ only exhibits
a small kink, which suggests more regularity.

\begin{figure*}
\centering
\begin{tikzpicture}
\begin{groupplot}
[
    small,
    no markers,
    grid = major,
    group style={group size=2 by 1, horizontal sep=0.1\textwidth},
    every axis plot/.append style={very thick},
    xmin = {0},
    xmax = {1},
    ymin = {-0.15},
    ymax = {0.30},
    ytick = {{-0.1},{0},{0.1},{0.2},{0.3}},
    yticklabels = {{$-0.10$},{$0$},{$0.10$},{$0.20$},{$0.30$}},
]
\nextgroupplot
\addplot table[x=x,y=Ax]{tangential_lineplot.dat};
\addplot table[x=x,y=Ay]{tangential_lineplot.dat};
\addplot table[x=x,y=Az]{tangential_lineplot.dat};
\nextgroupplot[%
    legend to name=SharedLegend,
    legend style = {legend columns=3,/tikz/every even column/.append style={column sep=0.5cm}},
    xmin = {0},
    xmax = {1},
    ymin = {-0.05},
    ymax = {0.10},
    ytick = {{-0.05},{0},{0.05},{0.10}},
    yticklabels = {{$-0.05$},{$0$},{$0.05$},{$0.10$}},
]
\addplot table[x=x,y=Ax]{lineplot.dat};
\addplot table[x=x,y=Ay]{lineplot.dat};
\addplot table[x=x,y=Az]{lineplot.dat};
\addlegendentry{$A_1$};
\addlegendentry{$A_2$};
\addlegendentry{$A_3$};
\end{groupplot}
\path (group c1r1.north west) -- node[above=0.5cm]{\pgfplotslegendfromname{SharedLegend}} (group c2r1.north east);
\end{tikzpicture}
\caption{\label{fig:lineplots} Numerical approximations of two different vector
potentials for the same velocity field $\vv{U}=\tfrac{1}{2}(-x_2,x_1,0)^\top$ on
the domain $\Omega=(0,1)^3\setminus[0.1,0.8]^3$, plotted along the line
$(x_1,0.5,0.8)^\top$. Left: the tangential vector potential $\vv{A}_{\mathrm{T}}
\in\HOdiv\cap\Hcurl$ by Amrouche et al.\cite{amrouche1998} Especially $A_1$ and
$A_2$ show very steep gradients at $x_1=0.1$, and furthermore exhibit jump-type
discontinuities. Further discontinuities away from $x_1\in\lbrace0.1,0.8\rbrace$
are due to the finite element approximation. Right: the new vector potential
$\vv{A}$ presented in this work. In this case we have $\vv{A}\in\vv{H}^{\frac{3}%
{2}}(\Omega)$; only a small kink in the components is visible at $x_1=0.1$ and
$x_1=0.8$. Also note the different scales.}
\end{figure*}

\section{Conclusions and Outlook}
In this work, we have established precise conditions under which a divergence-free
velocity field $\vv{U}\in\LebVec$ can be recovered from its given curl and
boundary data $\vv{U}\cdot\vv{n}$. Additionally, minor complementary assumptions
on the boundary data guarantees that this velocity field can be represented in
terms of a stream function $\vv{A}\in\vv{H}^1(\Omega)$, which can be explicitly
constructed. This stream function is more regular than the tangential vector
potential suggested by Amrouche et al.\cite{amrouche1998}

The regularity result of \Cref{thm:regularity} is sharp in several ways. Let us
for example consider the case of a handle-free domain $\Omega$ with $\beta_1=0$
and suppose that $\vv{F}\equiv\gv{0}$. It is then a classical result that the
velocity field $\vv{U}$ can be written in terms of the gradient of a scalar
potential: $\vv{U}=\grad P$, where $\Laplace P = 0$ in $\Omega$. Even if the
given boundary data $\vv{U}\cdot\vv{n}$ is smooth, it is known from the
regularity theory for the scalar Laplace equation that, on general Lipschitz
domains, the highest regularity one can expect is $P\in H^{\frac{3}{2}}(\Omega)$,
which therefore only leads to $\vv{U}\in\vv{H}^{\frac{1}{2}}(\Omega)$. For
any $\varepsilon>0$ there are indeed examples of domains $\Omega$ and boundary
data $\vv{U}\cdot\vv{n}$ where $P\notin H^{\frac{3}{2}+\varepsilon}(\Omega)$.
In this sense, the vector potential $\vv{A}\in\vv{H}^{\frac{3}{2}}(\Omega)$
introduced in this work has the highest possible regularity one can expect for
arbitrary Lipschitz domains.

However, an interesting question remains. Suppose that the given data $\vv{F}$
and $\vv{U}\cdot\vv{n}$ are such that the velocity field $\vv{U}$ does happen to
have higher regularity, say $\vv{U}\in\vv{H}^s(\Omega)$ for some $s\gg\frac{1}{2}$.
McIntosh and Costabel have proven that in this case, another vector potential
$\vv{A}_s\in\vv{H}^{1+s}(\Omega)$ exists.\cite[Corollary~4.7]{mcintosh2010} In
other words, there always exists a stream function that is more regular than its
velocity field by one order. In the numerical example discussed
in~\cref{sec:numerical-example}, such a smooth vector potential is given by:
\begin{equation}
\vv{A}_{\infty}(\vv{x}) = -\frac{1}{4}\begin{pmatrix}0\\0\\x_1^2 + x_2^2\end{pmatrix}.
\end{equation}
However, the numerical experiments indicate that the vector potential
proposed in this work is \emph{not smooth}. Therefore, the problem of devising
an algorithm to approximate reliably and efficiently the \enquote{smoothest possible}
stream function remains open. 

\section*{Acknowledgements}We would like to use the opportunity to express our
gratitude for the fruitful discussions with Ralf Hiptmair and his feedback on
previous versions of this manuscript. We would also like to thank the reviewers
for their comments, which greatly improved the quality of this manuscript,
lead us to investigating the case $\vv{F}\in\HOneDualVec$, and to expand the
section on numerical methods.

\bibliography{literature.bib}{}

\begin{thebibliography}{10}

\bibitem{cottet2000}
Cottet GH, Koumoutsakos PD. {\it Vortex Methods. Theory and Practice.}
\newblock Cambridge University Press; 2000.

\bibitem{majda2001}
Majda AJ, Bertozzi AL. {\it Vorticity and Incompressible Flow}.
\newblock {Cambridge Texts in Applied Mathematics.\ }Cambridge University
  Press; 2001.

\bibitem{john2017}
John V, Linke A, Merdon C, Neilan M, Rebholz LG. On the Divergence Constraint
  in Mixed Finite Element Methods for Incompressible Flows.  {\it SIAM Review.
  }2017;59(3):492--544.

\bibitem{hairer2006}
Hairer E, Lubich C, Wanner G. {\it Geometric Numerical Integration.
  Structure-Preserving Algorithms for Ordinary Differential Equations}.
\newblock No.~31 in {Springer Series in Computational Mathematics.\ }Springer;
  2nd~ed.2006.

\bibitem{raviart1986}
Girault V, Raviart PA. {\it Finite Element Methods for Navier--Stokes
  Equations. Theory and Algorithms}.
\newblock No.~5 in {Springer Series in Computational Mathematics.\ }Springer;
  1986.

\bibitem{amrouche1998}
Amrouche C, Bernardi C, Dauge M, Girault V. Vector Potentials in
  Three-dimensional Non-smooth Domains.  {\it Mathematical Methods in the
  Applied Sciences. }1998;21(9):823--864.

\bibitem{arnold2018}
Arnold DN. {\it Finite Element Exterior Calculus}.
\newblock No.~93 in {CBMS-NSF Regional Conference Series in Applied
  Mathematics.\ }Society for Industrial and Applied Mathematics; 2018.

\bibitem{alonso2010}
Alonso~Rodr\'iguez A, Valli R. {\it Eddy Current Approximation of Maxwell
  Equations. Theory, Algorithms and Applications}.
\newblock No.~4 in {Modeling, Simulation \& Applications.\ }Springer; 2010.

\bibitem{auchmuty2005}
Auchmuty G, Alexander JC. {$L^2$-well-posedness} of 3d div-curl boundary value
  problems.  {\it Quarterly of Applied Mathematics. }2005;(3):479--508.

\bibitem{kozono2009}
Kozono H, Yanagisawa T. {$L^r$}-variational Inequality for Vector Fields and
  the Helmholtz-Weyl Decomposition in Bounded Domains.  {\it Indiana University
  Mathematics Journal. }2009;58(4):1853--1920.

\bibitem{amrouche2013}
Amrouche C, Seloula NEH. {$L^p$}-Theory for Vector Potentials and {Sobolev's}
  Inequalities for Vector Fields: Application to the {Stokes} Equations with
  Pressure Boundary Conditions.  {\it Mathematical Models and Methods in
  Applied Sciences. }2013;23(1):37--92.

\bibitem{bramble2004}
Bramble JH, Pasciak JE. A new approximation technique for div--curl systems.
  {\it Mathematics of Computation. }2004;73(248):1739--1762.

\bibitem{alonso1996}
Alonso A, Valli A. Some remarks on the Characterization of the Space of
  Tangential Traces of {$\mathit{H}(\mathrm{rot};\Omega)$} and the Construction
  of an Extension Operator.  {\it {manuscripta mathematica}. }1996;89:159--178.

\bibitem{mclean2000}
McLean WCH. {\it Strongly Elliptic Systems and Boundary Integral Equations}.
\newblock Cambridge University Press; 2000.

\bibitem{sauter2011}
Sauter SA, Schwab C. {\it Boundary Element Methods}.
\newblock No.~39 in {Springer Series in Computational Mathematics.\ }Springer;
  2011.

\bibitem{buffa2002}
Buffa A, Costabel M, Sheen D. On traces for $\mathbf{H}(\mathbf{curl};\Omega)$
  in Lipschitz domains.  {\it Journal of Mathematical Analysis and
  Applications. }2002;276(2):845--867.

\bibitem{arnold2006}
Arnold DN, Falk RS, Winther R. Finite element exterior calculus, homological
  techniques, and applications.  {\it Acta Numerica. }2006;15:1--155.

\bibitem{hoermander1990}
H\"ormander LV. {\it The Analysis of Linear Differential Operators. Volume 1:
  Distribution Theory and Fourier Analysis}.
\newblock No.~256 in {Grundlehren der mathematischen Wissenschaften.\
  }Springer; 2nd~ed.1990.

\bibitem{dautray1990-i}
Dautray R, Lions JL. {\it Mathematical Analysis and Numerical Methods for
  Science and Technology. Volume 1: Physical Origins and Classical Methods}.
\newblock Springer; 1990.

\bibitem{pasciak2002}
Pasciak JE, Zhao J. Overlapping {Schwarz} Methods in
  {$\mathbf{H}(\textbf{curl})$} on Polyhedral Domains.  {\it Journal of
  Numerical Mathematics. }2002;10(3):221--234.

\bibitem{hiptmair2019}
Hiptmair R, Pechstein C. Regular Decompositions of Vector Fields: Continuous,
  Discrete, and Structure-Preserving.  In: Spectral and High Order Methods for
  Partial Differential Equations ICOSAHOM 2018:45--60; 2019.

\bibitem{claeys2019}
Claeys X, Hiptmair R. First-Kind Boundary Integral Equations for the
  Hodge--Helmholtz Operator.  {\it SIAM Journal on Mathematical Analysis.
  }2019;51(1):197--227.

\bibitem{buffa2001b}
Buffa A. Hodge decompositions on the boundary of nonsmooth domains: the
  multi-connected case.  {\it Mathematical Models and Methods in Applied
  Siences. }2001;11(9):1491--1503.

\bibitem{costabel1990}
Costabel M. A Remark on the Regularity of Solutions of Maxwell's Equations on
  Lipschitz Domains.  {\it Mathematical Methods in the Applied Sciences.
  }1990;12(4):365--368.

\bibitem{monk2003}
Monk P. {\it Finite Element Methods for Maxwell's Equations}.
\newblock {Numerical Mathematics and Scientific Computation.\ }Oxford
  University Press; 2003.

\bibitem{brezzi1991}
Brezzi F, Fortin M. {\it Mixed and Hybrid Finite Element Methods}.
\newblock No.~15 in {Springer Series in Computational Mathematics.\ }Springer;
  1991.

\bibitem{buffa2002b}
Buffa A, Costabel M, Schwab C. Boundary element methods for {M}axwell's
  equations on non-smooth domains.  {\it Numerische Mathematik.
  }2002;92(4):679--710.

\bibitem{kirchhart2020b}
Kirchhart M, Weniger D. Analytic Integration of the Newton Potential over
  Cuboids and an Application to Fast Multipole Methods.  {\it arXiv e-Prints.
  }2020;{ [math.NA] 2012.10304}.
\newblock \url{https://arxiv.org/abs/2012.10304}.

\bibitem{mcintosh2010}
McIntosh A, Costabel M. On {Bogovski\u\i} and regularized {Poincar\'e} integral
  operators for {de~Rham} complexes on {Lipschitz} domains.  {\it Mathematische
  Zeitschrift. }2010;265(2):297--320.

\end{thebibliography}
\end{document}